\documentclass[twoside]{article}
\usepackage{tikz-cd}

\usepackage[accepted]{aistats2021}

\usepackage[round]{natbib}

\usepackage{hyperref}       

\hypersetup{ 
	pdfkeywords={},
	pdfborder=0 0 0,
	pdfpagemode=UseNone,
	colorlinks=true,
	linkcolor=blue, 
	citecolor=blue, 
	filecolor=blue, 
	urlcolor=blue, 
	pdfview=FitH,
	pdfauthor={Anonymous},
}
\usepackage{mathtools}

\usepackage{url}            
\usepackage{booktabs}       
\usepackage{amsfonts}       
\usepackage{nicefrac}       
\usepackage{microtype}      
\usepackage{wrapfig}

\usepackage{amsmath,amsfonts,amssymb,amsthm,array}
\usepackage{algorithm}
\usepackage{algorithmic}%
\usepackage{bm}
\usepackage{color}
\usepackage{graphicx}

\usepackage{rotating}
\usepackage{amssymb}

\newcommand{\Exp}{\mathbb{E}}

\newcommand{\E}[1]{{\mathbb{E}\left[#1\right] }}    
\newcommand{\EE}[2]{{\mathbb{E}_{#1}\left[#2\right] }} 
\newcommand{\Var}[1]{\mathbf{Var}\left[#1\right]}
\newcommand{\Prob}[1]{\mathbb{P} \left[ #1\right]}
\newcommand{\R}{\mathbb{R}}
\newcommand{\N}{\mathbb{N}}
\DeclareMathOperator{\argmin}{arg\,min}

\usepackage{multirow}

\newcommand{\eqdef}{:=}

\newcommand{\cD}{{\cal D}}

\newcommand{\cL}{{\cal L}}

\newcommand{\cO}{{\cal O}}

\newcommand{\cX}{{\cal X}}

\newcommand{\mA}{{\bf A}}

\newcommand{\mR}{{\bf R}}

\usepackage{mdframed} 
\usepackage{thmtools}

\definecolor{shadecolor}{gray}{0.90}
\declaretheoremstyle[
headfont=\normalfont\bfseries,
notefont=\mdseries, notebraces={(}{)},
bodyfont=\normalfont,
postheadspace=0.5em,
spaceabove=1pt,
mdframed={
  skipabove=4pt,
  skipbelow=4pt,
  hidealllines=true,
  backgroundcolor={shadecolor},
  innerleftmargin=3pt,
  innerrightmargin=3pt}
]{shaded}

\declaretheorem[style=shaded,within=section]{definition}
\declaretheorem[style=shaded,sibling=definition]{theorem}
\declaretheorem[style=shaded,sibling=definition]{proposition}
\declaretheorem[style=shaded,sibling=definition]{assumption}
\declaretheorem[style=shaded,sibling=definition]{corollary}

\declaretheorem[style=shaded,sibling=definition]{lemma}

\providecommand{\trace}[1]{{\rm Trace}\left( #1\right)}
\providecommand{\norm}[1]{\left\| #1\right\|}
\newcommand{\dotprod}[1]{\left< #1\right>}

\usepackage{hyperref}

\usepackage[colorinlistoftodos,bordercolor=orange,backgroundcolor=orange!20,linecolor=orange,textsize=scriptsize]{todonotes}

\begin{document}

\twocolumn[

\aistatstitle{SGD for Structured Nonconvex Functions: Learning Rates, Minibatching and Interpolation}

\aistatsauthor{ Robert M. Gower$^*$\And Othmane Sebbouh \And  Nicolas Loizou$^*$ }

\aistatsaddress{ LTCI, T\'el\'ecom Paris, \\
Institut Polytechnique de Paris \And  ENS Paris,\\ CREST-ENSAE, CNRS \And Mila and DIRO, \\Universit\'{e} de Montr\'{e}al} ]

\begin{abstract}
Stochastic Gradient Descent (SGD) is being used routinely for optimizing non-convex functions. Yet, the standard convergence theory for SGD in the smooth non-convex setting gives a slow sublinear convergence to a stationary point. In this work, we provide several convergence theorems for SGD showing convergence to a global minimum for non-convex problems satisfying some extra structural assumptions. In particular, we focus on two large classes of structured non-convex functions: (i) Quasar (Strongly) Convex functions (a generalization of convex functions) and (ii) functions satisfying the Polyak-Lojasiewicz condition (a generalization of strongly-convex functions).
Our analysis relies on an \emph{Expected Residual} condition which we show is a strictly weaker assumption than previously used growth conditions, expected smoothness or bounded variance assumptions. We provide theoretical guarantees for the convergence of SGD for different step-size selections including constant, decreasing and the recently proposed stochastic Polyak step-size. 
In addition, all of our analysis holds for the arbitrary sampling paradigm, and as such, we give insights into the complexity of minibatching and determine an optimal minibatch size. Finally, we show that for models that interpolate the training data, we can dispense of our Expected Residual condition and give state-of-the-art results in this setting. \\[0.6cm]
\end{abstract}
\section{INTRODUCTION}
We consider the unconstrained finite-sum optimization problem
\begin{equation}
\label{MainProb}
\min_{x\in\R^d} \left[ f(x) = \frac{1}{n} \sum_{i=1}^n f_i(x)\right].
\end{equation} 
We use $\cX^* \subset \R^d$ to denote the set of minimizers $x^*$ of~\eqref{MainProb} and assume that $\cX^*$ is not empty and that $f(x)$ is lower bounded.  This problem is prevalent in machine learning tasks where~$x$ corresponds to the model parameters, $f_i(x)$ represents the loss on the training point $i$ and the aim is to minimize the average loss $f(x)$ across training points. 

When $n$ is large, stochastic gradient descent (SGD) and its variants are the preferred methods for solving \eqref{MainProb} mainly because of their cheap per iteration cost. The standard convergence theory for SGD~\citep{robbins1951stochastic, NemYudin1978, NemYudin1983book, Pegasos, Nemirovski-Juditsky-Lan-Shapiro-2009, lowernonconvex19, HardtRechtSinger-stability_of_SGD} in the smooth nonconvex setting shows slow sub-linear convergence to a stationary point. 
Yet in contrast, when applying SGD to many practical nonconvex problems of the form~\eqref{MainProb} such as matrix completion~\citep{pmlr-v37-sa15}, deep learning~\citep{ma2018power}, and phase retrieval~\citep{ShuoTanVershynin2019} the iterates converge globally, and sometimes, even linearly.  This is because these problems often have additional structure and properties, such as all local minimas are global  minimas~\citep{pmlr-v37-sa15,Kawaguchi2016}, the model interpolates the data~\citep{ma2018power} or the function under study is unimodal on all lines through a minimizer~\citep{hinder2019near}. By exploiting these structures and properties one can prove significantly tighter  convergence bounds.

Here we present a 
general analysis of SGD for two large classes of structured nonconvex functions: (i) the Quasar (Strongly) Convex functions and (ii) functions satisfying the Polyak-Lojasiewicz (PL) condition. In all of our results we provide convergence guarantees for SGD to the \emph{global minimum.} We also develop several corollaries for functions that interpolate the data.
\subsection{Background and Main Contributions}

\textbf{Classes of structured nonconvex functions.} The last few years has seen an increased interest in exploiting additional structure prevalent in large classes of nonconvex functions. Such conditions include error bound properties~\citep{fabian2010error},  essential strong convexity~\citep{liu2014asynchronous}, quasi strong convexity~\citep{Necoara-2018, gower2019sgd}, the restricted secant inequality~\citep{zhang2013gradient}, and the quadratic growth (QG) condition~\citep{anitescu2000degenerate, loizou2019randomized}. 
We focus on two of the weakest conditions: the quasar (strongly) convex functions \citep{hinder2019near, hardt2018dynam,Guminov2017} and functions satisfying the PL condition~\citep{polyak1987introduction,lojasiewicz1963topological,karimi2016linear}. 
The class of quasar-convex functions include all convex functions as a special case, but it also includes several nonconvex functions. Recently there is also some evidences suggesting that the loss function of neural networks have a quasar-convexity structure \citep{zhou2018sgd, kleinberg2018alternative}. 

\noindent \textit{Contributions.}  We show that SGD converges at a $\cO(1/\sqrt{k})$ rate on the \emph{quasar-convex functions} and prove linear convergence to a neighborhood for PL functions without any bounded variance assumption or growth assumptions on the stochastic gradients.  Instead, we rely on the recently introduced \emph{expected residual} (ER) condition~\citep{GowerRichBach2018}. 

\textbf{Assumptions on the gradient.} The standard convergence analysis of SGD in the nonconvex setting relies  on the bounded  gradients assumption $\Exp_i\|\nabla f_i(x^k)\|^2 <c$  \citep{recht2011hogwild, hazan2014beyond, rakhlin2012making} or  a growth condition $\Exp_{i}\|\nabla f_{i} (x^k)\|^2 \leq c_1 + c_2 \Exp\|\nabla f(x^k)\|^2$ \citep{Bertsekas:neurodynamic, bottou2018optimization, schmidt2017minimizing}. There is now a line of recent works \citep{pmlr-v80-nguyen18c, vaswani2018fast, gower2019sgd, Khaled-nonconvex-2020, lei2019stochastic, koloskova2020unified, loizou2020stochastic} which aims at relaxing these assumptions.

\noindent \textit{Contributions.} We use the recently introduced Expected Residual (ER) condition~\citep{GowerRichBach2018}.
We give the first convergence proofs for SGD under the \ref{eq:expresidual} condition and we show that \ref{eq:expresidual} is a strictly weaker assumption than the Strong Growth Condition (SGC)~\citep{schmidt2013fast}, Weak Growth (WGC)~\citep{vaswani2018fast} or the Expected Smoothness (ES)~\citep{gower2019sgd} assumptions.
Furthermore,  we show that the  \ref{eq:expresidual} condition holds for a large class of nonconvex functions including 1) smooth and interpolated functions 2) smooth and \emph{$x^*$-- convex} functions\footnote{The $x^*$-- convexity includes all convex functions and several nonconvex functions.}. Not only does the ER assumption hold for a larger class of functions, our resulting convergence rates under ER either match or exceed the state-of-the-art for quasar-convex and PL functions.

\textbf{PL condition.} The  PL condition \citep{polyak1987introduction,lojasiewicz1963topological} was introduced as a sufficient condition for the linear convergence of Gradient Descent for nonconvex functions. Assuming bounded gradients, it was shown in \cite{karimi2016linear} that SGD with a decreasing step size converges sublinearly at a rate of $\cO(1/\sqrt{k})$ for PL functions. In contrast,  by using a step size which depends on the total number of iterations, the same convergence rate can be achieved  without the need for the bounded gradient assumption~\citep{Khaled-nonconvex-2020}. Assuming in addition the interpolation condition and  SGC~\cite{vaswani2018fast} showed that SGD converges linearly for PL functions, but the specialization of this last result to gradient descent results in a suboptimal dependence on the condition number\footnote{Theorem~4 in~\cite{vaswani2018fast} specialized to GD gives a rate of $\mu^2/L^2$ where $L$ is the smoothness constant and $\mu$ the PL constant.} of the function. 

\noindent \textit{Contributions.} We provide a complete  minibatch analysis of SGD for PL functions which recovers the best known dependence on the condition number for Gradient Descent \citep{karimi2016linear} while also matching the current state-of-the-art rate derived in \cite{vaswani2018fast,lei2019stochastic} for SGD for interpolated functions. All of which relies on the weaker \ref{eq:expresidual} condition.
Moreover, we propose a switching step size scheme similar to \cite{gower2019sgd} which does not require knowledge of the last iterate of the algorithm. Using this step size, we prove that SGD converges sublinearly at a rate of $\cO(1/k)$ for PL functions 
without any additional bounded gradient of bounded variance assumption or growth assumption.

\textbf{Step-size selection for SGD.} The most important parameter that one should select to guarantee the convergence of SGD is the step-size or learning rate. There are several choices that one can use including constant step-size \citep{moulines2011non,needell2014stochastic,gower2019sgd, batchSGDNW16, pmlr-v80-nguyen18c}, decreasing step-size \citep{robbins1951stochastic, ghadimi2013stochastic, gower2019sgd,Nemirovski-Juditsky-Lan-Shapiro-2009, karimi2016linear} and adaptive step-size \cite{duchi2011adaptive, liu2019variance, kingma2014adam, bengio2015rmsprop, vaswani2019painless, ward2019adagrad}.

\noindent \textit{Contributions.}  We provide convergence theorems for SGD under several step-size rules  for minimizing quasar-convex functions and functions satisfying the PL condition, including constant and decreasing step-sizes and a recently introduced adaptive learning rate called the stochastic Polyak step-size~\citep{loizou2020stochastic}.

\textbf{Over-parameterized models and Interpolation.} Recently it was shown that SGD converges considerably faster when the underlying model is sufficiently over-parameterized as to interpolate the data. This includes problems such as deep matrix factorization~\citep{rolinek2018l4, rahimi2017reflections}, binary classification using kernels~\citep{loizou2020stochastic}, consistent linear systems \citep{gower2015randomized,richtarik2020stochastic,loizou2017momentum,loizou2019convergence} and multi-class classification using deep networks~\citep{vaswani2018fast, loizou2020stochastic}.

\noindent \textit{Contributions.} As a corollary of our main theorems we show that for models that interpolate the training data, we can further relax our assumptions, dispense of the \ref{eq:expresidual} condition altogether and instead,  simply assume that each $f_i$ is smooth.  Our results here match the  state-of-the-art convergence results~\citep{vaswani2018fast} but again under strictly weaker assumptions.

\subsection{SGD and Arbitrary Sampling}
We assume we are given access to unbiased estimates $g(x) \in \R^d$ of the gradient such that $\E{g(x)}  = \nabla f(x).$ 
For example, we can use a minibatch to form an estimate of the gradient such as $g(x) = \frac{1}{b}\sum_{i\in B}\nabla f_i(x),$
where $B \subset \{1,\ldots, n\}$ will be chosen uniformly at random and $|B|=b.$ 
To allow for any form of minibatching we use the \emph{arbitrary sampling} notation
\begin{equation}
g(x) = \nabla f_v(x) \eqdef \frac{1}{n} \sum _{i=1}^n v_i \nabla f_i(x),
\end{equation}
where $v\in\R^n_+$ is a random \emph{sampling vector} such that $\E{v_i}  = 1, \,\mbox{for }i=1,\ldots, n$ and $f_v(x)~\eqdef~\frac{1}{n}\sum_{i=1}^n v_i f_i(x)$.
Note that it follows immediately from this definition of sampling vector that $\E{g(x)} =\frac{1}{n} \sum _{i=1}^n \E{v_i} \nabla f_i(x) = \nabla f(x).$ 
In this work we mostly focus on the $b$--minibatch sampling, however we highlight that our analysis holds for every form of minibatching. 
\begin{definition}[Minibatch sampling]\label{def:minibatch}
Let $b \in [n]$. We say that $v \in \R^n$ is a $b$--minibatch sampling if
for every subset $S \in [n]$ with $|S| =b$ we have that
\[\Prob{v=\frac{n}{b}\sum_{i \in S} e_i}=\left.1 \right/\binom{n}{b} \eqdef \frac{b!(n-b)!}{n!}\]
\end{definition}
By using a double counting argument you can show that if $v$ is a $b$--minibatch sampling, it is also a valid sampling vector ($\E{v_i}  = 1$)~\citep{gower2019sgd}. See~\cite{gower2019sgd} for other choices of sampling vectors $v$.

With an unbiased estimate of the gradient $g(x)$, we can now use 
Stochastic gradient descent (SGD)  to solve~\eqref{MainProb} by sampling $g(x^k)$ i.i.d and  iterating
 \begin{eqnarray}  
\boxed{ x^{k+1} =  x^k- \gamma^k g(x^k)}  \label{eq:sgdstep}
\end{eqnarray} 
We also make the following mild assumption on the  gradient noise. 
\begin{assumption} \label{ass:grad-noise} 
The {\em gradient noise} $\sigma^2$ is finite.
 $$\sigma^2 \eqdef  \sup_{x^* \in \cX^*} \E{\norm{g(x^*)}^2} < \infty .$$
\end{assumption}

\section{CLASSES OF STRUCTURED NONCONVEX FUNCTIONS}
\label{SectionClassesFunctions}
We work with two classes of nonconvex problems: the quasar-convex functions and the functions that satisfy the Polyak-Lojasiewicz (PL) condition. 
\begin{definition}[Quasar  convex]
Let $\zeta \in (0, \,1]$ and  $x^* \in\cX^*$. We say that $f$ is $\zeta$- quasar-convex with respect to $x^*$ if for all $x \in \R^n$,
\begin{equation}
\label{eq:quasar-convex}
f(x^*) \geq f(x) +\frac{1}{\zeta} \langle \nabla f(x), x^*-x\rangle.
\end{equation}

\end{definition}
For shorthand we write $f \in QC(\zeta)$ to mean~\eqref{eq:quasar-convex}. The class of quasar-convex functions are parameterized by a positive constant $\zeta \in (0,\, 1]$. In the case that $\zeta=1$ then~\eqref{eq:quasar-convex}  is known as star convexity~\citep{nesterov2006cubic} (generalization of convexity). One can think of $\zeta$ as the value that controls the non-convexity of the function. As  $\zeta$ becomes smaller the function becomes ``more nonconvex" \citep{hinder2019near}.

One of weakest possible assumptions that guarantee a global convergence of gradient descent to the global minimum is the PL condition \citep{karimi2016linear}. Indeed, all local minimas of a function satisfying  the PL condition are also global minimas.
\begin{definition}[Polyak-Lojasiewicz (PL) Condition]
There exists $\mu >0$ such that
\begin{equation}\label{eq:PL}
\|\nabla f(x)\|^2 \geq 2\mu \left[f(x)-f^*\right]
\end{equation}
We write $f \in PL(\mu)$ if function $f$ satisfies~\eqref{eq:PL}.
\end{definition}
In addition we will also consider in several corollaries the following interpolation condition.
\begin{assumption}\label{ass:over}
We say that the interpolation condition holds if there exists $x^* \in \cX^*$ such that
 \begin{equation} \label{eq:interpolated} \min_{x\in\R^n} f_i(x) = f_i(x^*) \quad \mbox{for}\quad i=1,\ldots, n.\end{equation}
 \end{assumption}
  This interpolation condition has drawn much attention recently because many overparametrized deep neural networks
achieve a zero loss over all training data points~\citep{ma2018power} and thus  satisfy~\eqref{eq:interpolated}.

\section{EXPECTED RESIDUAL (ER)}
\label{SectionESER}
In all of our analysis of SGD we rely on the
   \emph{Expected Residual (ER)} assumption. In this section we formally define ER, provide new sufficient conditions for it to hold and relate it to the existing gradient assumptions.

ER measures how far the gradient estimate $g(x)$ is from the true gradient in the following sense.
\begin{assumption}[Expected residual]\label{ass:expresidual} 
We say that the \ref{eq:expresidual} condition holds or $g \in \text{ER}(\rho)$ if
\begin{align*}
\label{eq:expresidual}
\E{\norm{g(x) -g(x^*) - ( \nabla f(x)-\nabla f(x^*)) }^2} \\  \leq 2\rho\left(f(x)-f(x^*) \right), \quad \forall x \in \R^d.\tag{ER}
\end{align*}
\end{assumption}
Note that \ref{eq:expresidual} depends on both how $g(x)$ is sampled and the properties of the $f(x)$ function.

As a direct consequence of Assumption~\ref{ass:expresidual} we have the following bound on the variance of $g(x).$
\begin{lemma}
\label{lem:varbndrho}
If $g \in  \text{ER}(\rho)$ then 
\begin{equation}
\label{eq:varbndrho2}
\E{\|g (x)\|^2 } \leq  4  \rho ( f(x)-f^* ) + \|\nabla f (x)\|^2 +2 \sigma^2.
\end{equation}
\end{lemma}
It is this bound on the variance~\eqref{eq:varbndrho2} that we use in our proofs and allows us to avoid the stronger bounded gradient or bounded variance assumptions.

\textbf{Connections to other Assumptions.} Let us provide some more familiar sufficient conditions which guarantee that the \ref{eq:expresidual} condition holds. In doing so, we will also provide simple and informative bounds on the  expected residual constant $\rho$ when using minibatching.

We say that $f_i$ is $L_i$--smooth if $\forall x,z\in\R^d$ holds that:
\begin{align} \label{eq:Limain}
 f_i(z) -f_i(x) &\;\leq\;  \dotprod{\nabla f_i(x), z-x} +\frac{L_i}{2}\norm{z-x}^2.
\end{align}
Let $L_{\max} \eqdef \max_{i=1,\ldots, n} L_i. $ For $x^* \in \cX^*$, we say that $f_i$ is $x^*$--convex if
\begin{align}
f_i(x^*) -f_i(x) & \;\leq\;  \dotprod{\nabla f_i(x^*), x^*-x}, \quad  \forall x\in\R^d.\label{eq:convmain}
\end{align}
These two assumptions are sufficient for the $ER(\rho)$ condition to hold and give a useful bound on $\rho$, as we show in the following proposition.

\begin{proposition}
\label{prop:bniceconst} Let $v$ be a sampling vector.
If $f_i$ is $L_{i}$--smooth  and there exists $x^*\in\cX^*$ such that $f_i$ is $x^*$--convex  then  $g \in \text{ER}(\rho)$.  
If in addition $v$ is the $b$--minibatch sampling then 
\begin{equation} \label{eq:bniceconst}
\rho(b)=L_{\max} \frac{n-b}{(n-1)b}, \quad \sigma^2(b)=\frac{1}{b} \frac{n-b}{n-1} \sigma_1^2,
\end{equation}
where $\sigma_1^2 \eqdef   \sup_{x^* \in \cX^*} \frac{1}{n} \sum_{i=1}^n \norm{\nabla f_i(x^*)}^2.$
\end{proposition}
The bounds in Proposition~\ref{prop:bniceconst} have been proven before but under the stronger assumption that each $f_i$ is convex\footnote{See  Proposition 3.10 item (iii) in \cite{gower2019sgd} and Lemma  F.3  in~\cite{sebbouh2019towards}.}. In this work by dropping the requirement that each $f_i$ is convex we are able to consider interesting classes of nonconvex functions.

Indeed, the following theorem establishes that only smoothness and the interpolation condition are sufficient for the \ref{eq:expresidual} to hold.  Furthermore, we place the \ref{eq:expresidual} within a hierarchy of the following  assumptions used in analysing SGD for smooth nonconvex functions:

\emph{SGC: Strong Growth Condition} ( $\rho_{SGC}>0$) 
\begin{equation}\label{eq:SGC-main}
\E{\norm{g(x)}^2} \leq \rho_{SGC}\norm{\nabla f(x)}^2.
\end{equation}
\emph{WGC: Weak Growth Condition}($\rho_{WGC}>0$) 
\begin{equation}\label{eq:WGC-main}
\E{\norm{g(x)}^2} \leq 2\rho_{WGC}(f(x) -f(x^*)).
\end{equation}
\emph{ES: Expected Smoothness} ($\cL>0$)
\begin{equation}\label{eq:ES-main}
\E{\norm{g(x) -g(x^*)}^2} \leq 2\cL(f(x) -f(x^*)).
\end{equation}

Next in Theorem~\ref{theo:hierarchy} we show that the \ref{eq:expresidual} condition is (strictly) the weakest condition from the above list. 

\begin{theorem} \label{theo:hierarchy}
 Let  $ ES$, $WGC$ and $SGC$ denote Assumption 2.1 in~\cite{gower2019sgd}, Eq (7)  and Eq (2) in~\cite{vaswani2018fast}, respectively. Let  $L_i$  and $x^*$--convex abbreviate~\eqref{eq:Limain} and~\eqref{eq:convmain}, respectively. Then the following hierarchy holds,
\begin{tikzcd}
\boxed{SGC+L\mbox{--smooth}} \arrow[d]  \\
 \boxed{WGC}  \arrow[d] &  \boxed{L_i+\mbox{Interpolated}}  \arrow[d]\\
  \boxed{ES} \arrow[d] &  \boxed{ L_i+x^*\mbox{--convex} } \arrow[l]\\
  \boxed{\colorbox{blue!20}{ER}}
\end{tikzcd}
where  $L$--smooth is shorthand for function $f$ being $L$--smooth.
Finally, there are problems such that ER  holds and  ES \emph{does not} hold. Making ER the strictly weakest assumption among the above.
\end{theorem}

The important assumptions for analyzing SGD in the nonconvex setting are the ones that are downstream from $ L_i+\mbox{Interpolated}$. This is because there exists a rich class of nonconvex functions that are smooth and satisfy the interpolation condition. In contrast, the WGC is only known to hold for smooth and convex functions satisfying the interpolation assumptions (Proposition 2 in~\cite{vaswani2018fast}).

An important distinction between the ES~\eqref{eq:ES-main} and the~\ref{eq:expresidual} condition, is that~\eqref{eq:expresidual} always holds trivially for full batch sampling ($g(x) = \nabla f(x)$). In contrast ES may not hold. We found that this simple fact prevented us from obtaining the correct rates of convergence of SGD in the full batch setting (see Appendix~\ref{sec:convexpsmooth}).

In concurrent work,~\cite{Khaled-nonconvex-2020} propose an analysis of SGD for general smooth non-convex functions (and functions satisfying the PL condition\footnote{Under different step-size selection than the one we propose in our Theorems for PL functions.}) under the following ABC condition:
\paragraph{ABC.} Let $A,B,C\geq0$. We say that $ABC$ condition holds if
\begin{equation}\label{eq:ABCpaper}
\small{\E{\norm{g(x)}^2} \leq 2A(f(x) -f(x^*) +B\norm{\nabla f(x)}^2 + C.}
\end{equation}
We note that by properly choosing the constants $A$, $B$ and $C$ in the ABC condition we can recover the assumptions SGC, WGC, and ES appearing in Theorem~\ref{theo:hierarchy}. In Appendix~\ref{sec:hierarchy} we show how condition~\eqref{eq:varbndrho2} which is a consequence of~\ref{eq:expresidual} is also a special case of the ABC assumption. 

\section{CONVERGENCE ANALYSIS}
\label{SectionConvergenceAnalysis}
In this section, we present the main convergence results. 
Proofs of all key results can be found in the Appendix~\ref{AppendixProofs}. In Appendix~\ref{AppendixTheoryAdditional}, we present additional convergence results on quasar-strongly convex functions (Section~\ref{sec:strongquasar}) and on convergence under expected smoothness (Section~\ref{sec:convexpsmooth}).
\subsection{Quasar Convex functions}
\subsubsection{Constant and Decreasing Step-sizes}
Now we present our results for quasar-convex functions for SGD with a constant, finite horizon and  decreasing step sizes.

\begin{theorem} \label{theo:master-quasar-convex-res}
Assume $f(x)$ is $L$--smooth, $\zeta-$quasar-convex with respect to $x^*$ and $g \in ER(\rho)$. Let $0<\gamma_k<\frac{\zeta}{2\rho + L}$ for all $k \in \N$ and let $r_0 \eqdef \|x^0 - x^*\|^2$. Then iterates of SGD given by \eqref{eq:sgdstep} satisfy:
\begin{multline}\label{eq:master-quasar-convex-res}
\min_{t=0,\dots,k-1} \E{f(x^t) -f(x^*)} \\ \leq \frac{1}{\sum_{i=0}^{k-1}\gamma_i(\zeta - \gamma_i(2\rho + L))} \left[\frac{r^0}{2} + \sigma^2 \sum_{t=0}^{k-1}\gamma_t^2 \right].
\end{multline}
Moreover, for  $\gamma < \frac{\zeta}{2\rho + L}$ we have that\\
{\bf 1.} 
	If $\forall k \in \N, \; \gamma_k = \gamma \equiv \frac{1}{2} \frac{\zeta}{(2\rho + L)}$ then $\forall k \in \N$,
		\begin{equation}\label{eq:master-quasar-const}
\min_{t=0,\dots,k-1}\E{f(x^t) - f(x^*)} \leq 2r_0\frac{2\rho + L}{ \zeta^2 k} + \frac{\sigma^2}{2\rho + L} .
	\end{equation}
{\bf 2.} Suppose SGD~\eqref{eq:sgdstep} is run for $T$ iterations. If $\forall k=0,\dots,T-1, \; \gamma_k = \frac{\gamma}{\sqrt{T}}$ then
	\begin{equation}
	\label{canojaksdk}
\min_{t=0,\dots,T-1}\E{f(x^t) - f(x^*)} \leq \frac{r_0 + 2\gamma^2\sigma^2}{\gamma\sqrt{T}}.
	\end{equation}
{\bf 3.}  If $\forall k \in \N, \; \gamma_k = \frac{\gamma}{\sqrt{k+1}}$ then $\forall k \in \N$,
		\begin{multline}\label{eq:quasarconvdecrease}
\min_{t=0,\dots,k-1}\E{f(x^t) - f(x^*)} \\ \quad \leq  \frac{1}{4\gamma}\frac{r_0 + 2\gamma^2\sigma^2(\log(k)+1)}{ \zeta(\sqrt{k} - 1)-\gamma(\rho + L/2)(\log(k) + 1)},
	\end{multline}
which is a convergence  rate of~$\cO\left(\frac{\log(k)}{\sqrt{k}}\right)$.
\end{theorem}
To the best of our knowledge, the only prior result for the convergence of SGD for smooth quasar-convex functions 
was a finite horizon result similar to~\eqref{canojaksdk} but under the strong assumption of bounded gradient variance~\citep{hardt2018dynam}.
Of particular importance is~\eqref{eq:quasarconvdecrease} which is the first $\cO\left(\log(k)/\sqrt{k}\right)$ any time convergence rate for quasar-convex functions. Indeed, this rate has only been achieved before under the \emph{strictly stronger assumption} that the $f_i$'s are smooth, convex and $g(x)$ has bounded variance~\citep{Nemirovski-Juditsky-Lan-Shapiro-2009}. Indeed, strictly stronger since due to Theorem~\ref{theo:hierarchy} the \ref{eq:expresidual} condition holds when the $f_i$'s are smooth and convex without any bounded gradient assumption.

When considering interpolated functions, we can completely drop the \ref{eq:expresidual} condition due to Theorem~\ref{theo:hierarchy}. In this next corollary we highlight this and show how  the complexity of SGD is affected by increasing the minibatch size.
\begin{corollary} \label{cor:quasar-minib} Let $f$ be $\zeta$-quasar-convex with respect to $x^*$. Let the interpolation Assumption~\ref{ass:over} hold and let each  $f_i$ be $L_i$--smooth.
 If $v$ is a $b$-minibatch sampling and 
 $\gamma_k \equiv \frac{1}{2} \frac{\zeta(n-1)b}{2L_{\max} (n-b)+ L(n-1)b}$ then 
\begin{align} \label{eq:SGDforStarOVerER}
\min_{t=0,\dots,k-1}&\E{f(x^t) - f(x^*)} \nonumber \\ 
\quad \leq &\frac{2L_{\max} (n-b) + L (n-1)b}{ \zeta^2 (n-1)b } \frac{2r_0}{k}.
\end{align}
This shows that $TC(b)$, the \emph{total complexity} as a function of the minibatch size, to bring $\underset{i=1,\dots,k-1}{\min}\E{f(x^i) -f^*} \leq \epsilon$ is given by
\begin{equation}\label{eq:totalcomplexStaroverER}
TC(b)\; \leq \; 
\frac{2 (n-b)L_{\max} +(n-1)b L }{\zeta^2(n-1)} \frac{2r_0}{\epsilon}.
\end{equation}
 Thus the optimal minibatch size $b^*$ that minimizes this total complexity 
is given by
\begin{equation}\label{eq:boptimalStaroverER}
b^*  = 
\begin{cases}
1 & \mbox{if } (n-1) \geq  2\frac{L_{\max}}{L}\\
n & \mbox{if } (n-1) < 2\frac{L_{\max}}{L}.
\end{cases}
\end{equation}
\end{corollary}
Specializing~\eqref{eq:SGDforStarOVerER} to the full batch setting ($n=b$), we have that gradient descent (GD) with step size $\gamma = \frac{\zeta}{4L} $ converges as follows\footnote{Here we use that the smoothness of $f$ guarantees that $f(x^1),\ldots, f(x^t)$ for GD is a decreasing sequence.}: $f(x^t) - f(x^*) \leq \frac{2L\norm{x^0 - x^*}^2}{\zeta^2 k}.$ This is exactly the rate given recently for GD for quasar-convex functions in~\cite{Guminov2017}, with the exception that we have a squared dependency on $\zeta$ the quasar-convex parameter.
\subsubsection{Stochastic Polyak Step-size (SPS) - Guarantee Convergence without tuning}
The stochastic Polyak step size (SPS) is a recently proposed adaptive step size selection for SGD \citep{loizou2020stochastic}.
SPS is a natural extension of the classical Polyak step-size~\citep{polyak1987introduction} (commonly used in the deterministic subgradient method) to the stochastic setting.

In this work, we generalize the SPS to the arbitrary sampling regime and provide a novel convergence analysis of SGD with SPS for the class of smooth, quasar (strongly) convex functions.

Let $v$ be a sampling vector and let $f_v=\sum_{i=1}^n f_i(x) v_i$.  Let $f_v^* = \min_{x\in \R^n} f_v(x)$ which we assume exists. Just like the gradient, we have that $f_v$ is an unbiased estimate of $f$. Now given a sampling vector $v$,  we define the \emph{Stochastic Polyak Step-size} (SPS) as
\begin{equation}
\label{SPLRv}
\text{SPS:} \quad \gamma_{k} =\frac{f_v(x^k)-f_v^*}{c \, \|\nabla f_v(x^k)\|^2},
\end{equation}
where $0<c \in \R$.  As explained in~\cite{loizou2020stochastic}, the SPS rule is particularly effective when training over-parameterized models capable of interpolating the training data (when the interpolation Assumption~\ref{ass:over} holds). In this case, SGD with SPS converges to the exact minimum (not to a neighborhood of the solution)~\citep{loizou2020stochastic}. In addition, if $f_i^* \coloneqq \min_{x\in \R^n} f_i(x)$ then for machine learning problems using standard unregularized surrogate loss functions (e.g. squared loss for regression, hinge loss for classification) it holds that $f_i^*=0$~\citep{loizou2020stochastic}.
If on top of this, we assume that interpolation Assumption~\ref{ass:over} holds (that is, $f_i^*=f_i(x^*)$, $\forall i \in [n]$), then we have that $f_i^*=f_v^*=f_v(x^*)=0$ for every $i \in [n]$ and for every $ v$. 

By assuming that every $f_i$ is $L_i$--smooth, we have that  $f_v$ is $L_v$--smooth with $L_v \eqdef \frac{1}{n}\sum_{i=1}^n v_i L_i$. This smoothness combined with Lemma~\ref{lem:smoothsubopt} and Jensen's inequality  gives a lower bound on SPS~\eqref{SPLRv}: 
\begin{align}\label{eq:a8ejh8s434}
\frac{1}{2 c \E{L_v}} \overset{Jensen}{\leq}  \E{\frac{1}{2 c L_{v}} } \leq \E{\gamma_{k}=\frac{f_{v}(x^k)-f_v^*}{c \|\nabla f_v(x^k)\|^2} }.
\end{align}
This lower bound combining with the following new bound allows us to establish the forthcoming theorem for quasar-convex functions.
\begin{lemma}\label{lem:cLmax}
Assume interpolation~\ref{ass:over} holds. Let $f_i$ be $L_i$--smooth and let $v$ be a sampling vector. It follows that there exists $\cL_{\max} > 0 $ such that
\begin{equation} \label{eq:cLmaxbnd}
\frac{1}{2  \cL_{\max}} ( f(x) -f^* )  \;\leq \; \E{\frac{(f_v(x)-f_v^*)^2}{ \|\nabla f_v(x)\|^2}}.
\end{equation}
Furthermore,  for $B \subset\{1,\ldots, n\}$ 
let $L_B$ be the smoothness constant of $f_B \eqdef\frac{1}{b} \sum_{i\in B}  f_i$.  If $v$ is the $b$--minibatch sampling then 
\[\cL_{\max} \; =\; \cL_{\max}(b) \;=\; \underset{i=1,\ldots, n}{ \max} \dfrac{  \binom{n-1}{b-1}}{ \sum_{B: i \in B} L_B^{-1}}.\]
   \end{lemma}
With the above lemma we can now establish our main theorem.
\begin{theorem}
\label{theo:SPLquasarconvex}
Let $v$ be a sampling vector.
Assume interpolation~\ref{ass:over} holds. Assume that  each $f_i$ is $\zeta$-quasar-convex with respect to $x^*$ and  $L_i$-smooth.
Then SGD with $\text{SPS}$ \eqref{SPLRv} and $c> \frac{1}{2\zeta}$ converges as follows: 
$$\min_{i=0,\ldots, K-1}\E{f(x^i) - f^*}  \; \leq \; \frac{2c^2} {2c \zeta-1}\frac{ \cL_{\max}}{K} \|x^{0}-x^*\|^2, $$
where $\cL_{\max}$ is  defined in Lemma~\eqref{lem:cLmax}.
\end{theorem}
We now use  $\cL_{\max}(b)$ given in Lemma~\ref{lem:cLmax} to derive the importance sampling complexity. To the best of our knowledge, this is the first importance sampling result for SGD with SPS in any setting.
\begin{corollary}\label{cor:itercomplexSPLquasar}
Consider the setting of Theorem~\ref{theo:SPLquasarconvex} with $c=1/4\zeta.$ Given $\epsilon>0$ we have that 
\begin{multline}\label{eq:itercomplexSPLquasar}
 k \geq \frac{\cL_{\max}}{4 \zeta^2} \frac{\|x^{0}-x^*\|^2}{\epsilon} =\cO\left(\frac{\cL_{\max}}{\zeta^2 \epsilon} \right) \\ \; \Rightarrow \; \min_{i=0,\ldots, K-1}\E{f(x^i) - f^*} <\epsilon.
\end{multline}
{\bf 1.} (Full batch) If we use full batch sampling we have that $\cL_{\max} =L$ and~\eqref{eq:itercomplexSPLquasar} becomes $\cO(L/\epsilon\zeta^2)$
{\bf 2.} (Importance sampling). If we use single element sampling with $p_i =L_i/\sum_{j}L_j$ we have that $\cL_{\max} =\frac{1}{n} \sum_{j=1}L_j \eqdef \overline{L}$ and~\eqref{eq:itercomplexSPLquasar} becomes $\cO(\overline{L}/\epsilon\zeta^2)$.
\end{corollary}
We highlight that the result on importance sampling of Corollary~\ref{eq:itercomplexSPLquasar} requires the knowledge of the smoothness parameters $L_i$. This comes in contradiction with the parameter-free nature of the stochastic Polyak step-size. However, such result was missing from the literature and we believe that it could work as a first step towards the understanding of efficient (parameter-free) non-uniform sampling variants of SGD with SPS.  We leave such extensions for future work.
\subsection{PL Condition}
Here we present our convergence results for functions satisfying the PL condition~\eqref{eq:PL}.
\subsubsection{Constant Step-size}
Let us start by presenting convergence guarantees for SGD with constant step-size.
\begin{theorem}
\label{theo:PLConstant}
Let $f$ be $L$-smooth. Assume $f \in PL(\mu)$ and $ g \in \text{ER}(\rho)$. Let $\gamma_k = \gamma\leq \frac{1}{1 +2\rho/\mu} \frac{1}{L},$ for all $k$, then SGD given by \eqref{eq:sgdstep} converges as follows: 
\begin{equation}\label{eq:functionTheoremExpResidual}
\Exp[f(x^{k})-f^*] \leq \left(1- \gamma \mu \right)^k [f(x^0)-f^*] + \frac{L \gamma \sigma^2} {\mu}.
\end{equation}
Hence, given $\epsilon>0$ and using the step size $\gamma =\frac{1}{L}\min \left\{ \frac{\mu \epsilon}{2 \sigma^2}, \, \frac{1}{1 +2\rho/\mu}\right\}$ we have that
\begin{multline}\label{eq:itercomplexPL}
k\geq \frac{L} {\mu} \max \left\{ \frac{2 \sigma^2}{\mu \epsilon}, \, 1 +\frac{2 \rho}{\mu}\right\} \log\left(\frac{2(f(x^0)-f^*)}{\epsilon}\right) \\
\;\; \implies \;\; \E{f(x^k) -f^*} \leq \epsilon.
\end{multline}
\end{theorem}
When the function is able to interpolate the data (interpolation condition~\ref{ass:over} is satisfied), SGD with constant step size  convergences with a linear rate to the exact solution (no neighborhood of convergence), as we show next.
\begin{corollary}
\label{theo:SGDforPolyakOVer}
Consider the setting of Theorem~\ref{theo:PLConstant} and  assume interpolation~\ref{ass:over} holds. Then SGD with $\gamma_k = \gamma\leq \frac{1}{1 +2\rho/\mu} \frac{1}{L}$ converges linearly at a rate of $(1-\gamma \mu).$
Consequently for every  $\epsilon>0$,
 the iteration complexity  of SGD
to achieve $\E{f(x^k) -f^*} \leq \epsilon$ is
\begin{equation}\label{eq:itercomplexPLover}
k \;\geq  \;\frac{L} {\mu} \left(1 +2\frac{\rho}{\mu}\right) \log\left(\frac{f(x^0)-f^*}{\epsilon}\right).
\end{equation}
If $v$ is a $b$--minibatch sampling then $TC(b)$, the \emph{total complexity} with respect to the minibatch size, is 
\begin{equation}\label{eq:totalcomplexPLover}
TC(b)\; \leq \; \frac{L} {\mu} \left(b +2\frac{L_{\max} }{\mu}\frac{n-b}{n-1} \right) \log\left(\frac{f(x^0)-f^*}{\epsilon}\right).
\end{equation}
Finally, let $\kappa_{\max} \eqdef L_{\max}/\mu.$
The minibatch size $b^*$ that optimizes the total complexity is given by
\begin{equation}\label{eq:boptimalPLover}
b^*  = 
\begin{cases}
1 & \mbox{if } n-1 \geq  2\kappa_{\max}\\
n & \mbox{if } n-1 <  2\kappa_{\max}.
\end{cases}
\end{equation}
\end{corollary}
Note that Corollary~\ref{theo:SGDforPolyakOVer} recovers the linear convergence rate of the gradient descent algorithm under the PL condition~\citep{karimi2016linear} as a special case. Indeed for gradient descent we have that $\sigma = 0 = \rho$. Thus by choosing $ \gamma = \frac{1}{L}$ the resulting iteration complexity is $ \frac{L}{\mu}  \log(\epsilon^{-1})$ 
which is currently the tightest known convergence result for gradient descent under the PL condition~\cite{karimi2016linear}. 
On the other extreme, we see that for $b=1$, that is SGD without minibatching, we obtain the convergence
rate $1-\mu^2/3LL_{\max}$ which matches the current state-of-the-art rate~\cite[Thm. 4]{vaswani2018fast},~\cite[Thm. 3]{Khaled-nonconvex-2020} and~\cite[Thm. 4]{lei2019stochastic} known under the exact same assumptions. Thus  we recover the best known rate on either end ($b=n$ and $b=1$), and give the first rates for everything in between $1<b<n$. To the best of our knowledge our result is the first analysis of SGD for PL functions that recovers the deterministic gradient descent convergence as special case.

The closest work to our result, on the convergence of SGD for PL functions is \cite{Khaled-nonconvex-2020}. There the authors provide similar convergence result to Theorem~\ref{theo:PLConstant} but using different step-size selection and under the slightly more general ABC condition~\eqref{eq:ABCpaper}. In Appendix~\ref{sec:comparekhaled} we present a detailed comparison of our Theorem~\ref{theo:PLConstant} and Theorem~3 in~\cite{Khaled-nonconvex-2020}.

\subsubsection{Decreasing Step-size}

As an extension of Theorem~\ref{theo:PLConstant}, we
also show how to obtain a  $\cO(1/k)$ convergence for SGD using an insightful \emph{stepsize-switching rule}. This stepsize-switching rule describes when one should switch from a constant to a decreasing step-size regime. 

\begin{theorem}[Decreasing step sizes/switching strategy]
\label{TheoremPLDecreasing}
Let $f$ be an $L$-smooth. Assume $f \in PL(\mu)$ and $ g \in \text{ER}(\rho)$. Let $k^* \eqdef 2\frac{L}{\mu} \left(1+2\frac{\rho}{\mu}\right)$ and 
\begin{equation}
\gamma^k= 
\begin{cases}
\displaystyle  \frac{\mu}{L (\mu +2\rho)}, & \mbox{for}\quad k \leq \lceil k^* \rceil\\[0.3cm]
\displaystyle \frac{2k+1}{(k+1)^2 \mu} &  \mbox{for}\quad k >  \lceil k^* \rceil
\end{cases}
\end{equation}
If $k \geq  \lceil k^*  \rceil$, then SGD given by \eqref{eq:sgdstep} satisfies:
\begin{equation}\label{eq:decreasingstepPLExpResidual}
\Exp[f(x^{k})-f^*] \le  \frac{4 L \sigma^2 }{\mu^2 }\frac{1}{k} + \frac{( k^*)^2}{k^2 e^2}  [f(x^0)-f^*]  .
\end{equation}
\end{theorem}

\paragraph{Stochastic Polyak-Step-size (SPS).} For the convergence of SGD with SPS for solving functions satisfying the PL condition we refer the interested reader to Theorem 3.5 in \cite{loizou2020stochastic}. There the authors focus on analyzing SGD with single-element uniform sampling. By assuming interpolation, their convergence results can be trivially extended to the arbitrary sampling paradigm using the lower bound~\eqref{eq:a8ejh8s434} and Lemma~\ref{lem:cLmax}.

\section{EXAMPLES} 
\label{sec:examples}
In this section we provide some examples of classes of nonconvex functions that satisfy  the assumptions of our main theorems.

\paragraph{System Identification.}
In optimal control sometimes we need to learn the underlining dynamics of the system we are trying to control. For instance, consider the system governed by the \emph{linear dynamics}
\begin{align}
h_{t+1} & = A h_t + B w_t\\
y_t & = C h_t + Dw_t + \xi_t,  \label{eq:lined}
\end{align}
where $w_t\in \R$ and $y_t\in \R$ are the input and output at time $t$,  $h_t\in\R^d$ is the hidden state,  and $\xi_t \in \R$ is a random variable sampled i.i.d at each iteration. The parameters we want would to learn are the matrices $A \in \R^{d\times d},\, B\in \R^{d \times 1}$, $C\in\R^{1 \times d}$ and $D\in \R$ that govern the dynamics. Furthermore, we can only observe the input-output pairs $(w_t,y_t)$ by simulating the dynamics.

Our goal is to use the collected samples of the simulation $(w_t,y_t)$ to then \emph{fit} a linear model
 \begin{align}
h_{t+1} & = \hat{A} h_t + \hat{B} w_t \nonumber\\
\hat{y}_t & = \hat{C} h_t + \hat{D}w_t, \label{eq:ythat}
\end{align}
governed by the matrices $x \eqdef (\hat{A}, \hat{B}, \hat{C}, \hat{D})$ such that the output of our model $\hat{y}_t$, and that of the simulation $y_t$ are close. That is we want to solve
\begin{equation}
\min_{ x = (\hat{A}, \hat{B}, \hat{C}, \hat{D})}f(x) \eqdef \EE{w_t, \xi_t }{\frac{1}{T} \sum_{i=1}^T \norm{y_t - \hat{y}_t}^2}. \label{eq:lineardynobj}
\end{equation}
As done in~\cite{hardt2018dynam}, we assume that the states $w_t$ are sampled from some fixed distribution.

 This objective function~\eqref{eq:lineardynobj} is highly non-convex due to repeated multiplications of the parameters, as we can see by substituting out the hidden states and unrolling the recurrence~\eqref{eq:ythat} since 
 \begin{equation}
 \hat{y}_t \; = \; \hat{D} w_t + \sum_{k=t}^{t-1}\hat{C} \hat{A}^{t-k-1}\hat{B} w_k +\hat{C}\hat{A}^{t-1}h_0.
 \end{equation}
Despite this non-convexity, the objective function~\eqref{eq:lineardynobj} is quasar-convex~\eqref{eq:quasar-convex} and $L$--weakly smooth\footnote{To be precise the objective function is well approximated and upper bounded by a  quasar-convex and weakly-smoooth function, which also requires some domain restrictions. SGD is then applied to this upper bound. See~\citep{hardt2018dynam} for  details. }, that is
\begin{equation}
\norm{\nabla f(x)}^2 \; \leq\; 2 L (f(x) - f(x^*)) \label{eq:WS}\tag{WS}.
\end{equation}
By also bounding the domain of the parameters, \cite{hardt2018dynam} show 
 that the stochastic gradients $g(x)$ have bounded variance
\begin{equation}
\E{\norm{\nabla f(x) -g(x)}^2} \; \leq \; \sigma^2 \label{eq:BV} \tag{BV}.
\end{equation}
\cite{hardt2018dynam} then use quasar convexity,~\eqref{eq:WS} and~\eqref{eq:BV} to show that the linear dynamics~\eqref{eq:lined} can be learned with SGD in polynomial time.

As a consequence of~\cite{hardt2018dynam}  results, first we show that the objective function~\eqref{eq:lineardynobj} satisfies the assumptions of our Theorem~\ref{theo:master-quasar-convex-res}. 

\begin{theorem} \label{thm:BV-WS-ES}The following hierarchy holds
\vspace{-0.15cm}\begin{center}
\begin{tikzcd}
\boxed{\ref{eq:BV} + \ref{eq:WS}} \arrow[r] &
  \boxed{ES} \arrow[r] &
  \boxed{\colorbox{blue!20}{ER}}
\end{tikzcd}
\end{center}
Furthermore, there are functions for which~\eqref{eq:expresidual} holds and~\eqref{eq:BV} does not. 
\end{theorem} 
Consequently, since~\eqref{eq:lineardynobj} satisfies \eqref{eq:BV}, \eqref{eq:WS} and~\eqref{eq:quasar-convex} we have that it satisfies~\eqref{eq:expresidual} and~\eqref{eq:quasar-convex}, and thus by Theorem~\ref{theo:master-quasar-convex-res} SGD applied to~\eqref{eq:lineardynobj} converges at a rate of $O(1/\sqrt{t}).$

We conjecture that the linear dynamics~\eqref{eq:lined} could be learned without the bounded gradient assumption by only relying on the~\eqref{eq:expresidual} condition. 
This would be significant because, it would mean that the costly projection step onto the constrained set of parameters, required so that~\eqref{eq:BV} holds, may not be necessary. We leave this conjecture to be verified in future work.

\paragraph{Contrived Illustrative Example.}
To given an example of a visually non-convex functions that satisfies both the PL and \ref{eq:expresidual} condition we consider the separable functions $f(x) = \frac{1}{n} \sum_{i=1}^n f_i(x_i).$
 If each $f_i(x_i)$ satisfies the PL condition with constant $\mu_i$ then  $f(x)$ satisfies the PL condition with $\mu =\min_{i=1,\ldots, n} \frac{\mu_i}{n} .$
If in addition each $f_i$ is a smooth function  then according to Theorem~\ref{theo:hierarchy} we have that the~\ref{eq:expresidual} condition holds, and thus Theorem~\ref{theo:SGDforPolyakOVer} holds.
\begin{figure}
\centering
    \includegraphics[width=0.25\textwidth]{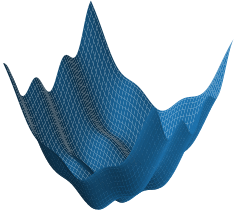} 
\caption{\footnotesize Surface plot of $x^2 + 3\sin^2(x) + 1.5 y^2 + 4\sin^2(y)$}
\label{fig:nonconvex} 
\end{figure}
As an example, consider the nonconvex function
\begin{equation} \label{eq:separable-exe}
f(x) = \frac{1}{n} \sum_{i=1}^n a_i (x_i^2 + 4 b_i \sin^2(x_i)) \eqdef f_i(x),
\end{equation}
where $a_i>0$ and  $1>b_i>0$ for $i=1,\ldots, n$, so that each $f_i$ satisfies the PL condition (see~\cite{karimi2016linear}\footnote{In~\cite{karimi2016linear} the authors claim  that $x^2 + 3\sin^2(x)$ is PL. We then used computer aided analysis to show that  $x^2 + 3b\sin^2(x)$ satisfies the PL condition for $0<b<4.$}). The function~\eqref{eq:separable-exe} is interpolated since $x^* = 0$ is a global minima for each $f_i$. Furthermore $f_i$ is smooth since $|f_i''(x)| \leq 2a_i +6b_i$. By the above arguments, so does $f$ satisfy the PL condition. Thus by Theorem~\ref{theo:SGDforPolyakOVer} we know that SGD converges linearly when applied to~\eqref{eq:separable-exe}.
To illustrate that such functions~\eqref{eq:separable-exe} are nonconvex, we have a surface plot for $n=2$ in Figure~\ref{fig:nonconvex}.
\paragraph{Nonlinear least squares.} Let $F: \R^d \rightarrow \R^n$ be a differentiable function where $DF(x) \in \R^{n \times d}$ is its Jacobian.
Now consider the nonlinear least squares problem $\min_{x \in \R^d} f(x) \eqdef \frac{1}{2n} \norm{F(x) -y}^2=\frac{1}{2n}  \sum_{i=1}^n (F_i(x) -y_i)^2 ,$ where $y \in \R^n.$ 
\begin{lemma} \label{lem:nonlsqPLexe}
Assume there exists $x^*\in \R^d$ such that $F(x^*) =y.$ 
If the $F_i(x)$ functions are Lipschitz and the  $DF(x)$ has full row rank then $F$ satisfies the PL and the \ref{eq:expresidual} condition. 
\end{lemma}

\paragraph{Star/quasar-convex.}
Several nonconvex empirical risk problems are quasar-convex functions~\citep{LeeV16}.
Let $f_i: \R^d \mapsto \R$ be a smooth star-convex (quasar-convex with $\zeta =1$) centered at $0$. Let $\mA \in  \R^{m \times n}, b\in \R^m$  such that there exists $\mA x^* =b.$ Since compositions of affine maps with star convex functions are star convex \cite[Section A.4]{LeeV16}  we have that 
$f_i(\mA x -b)$ is star convex centered at $x^*.$ Furthermore the average of star convex functions that share the same center are star convex. Thus, $ f(x) = \frac{1}{n}\sum_{i=1}^n f_i(\mA x -b),$
is a star-convex function which also satisfies the interpolation condition. 
\section{CONCLUSION}
\label{SectionConclusion}
We establish a hierarchy between the expected residual (\ref{eq:expresidual}) condition and a host of other assumptions previously used in the analysis of SGD in the smooth setting, showing that \ref{eq:expresidual} is a strictly weaker condition.
Using the \ref{eq:expresidual}, we present the first convergence results for SGD under different step-size selections (constant, decreasing, and stochastic Polyak step-size) on quasar-convex functions~\eqref{eq:quasar-convex} without the bounded gradient or bounded variance assumption. For functions satisfying the PL condition~\eqref{eq:PL} we provide tight theoretical convergence guarantees for minibatch SGD that recover the best-known convergence results for deterministic gradient descent and single-element sampling SGD as special cases, and all minibatch sizes in between.

\subsubsection*{Acknowledgements}
Nicolas Loizou acknowledges support by the IVADO post-doctoral funding program. 

The work of Othmane Sebbouh was supported in part by the French government under management of Agence Nationale de la Recherche as part of the "Investissements d’avenir" program, reference ANR19-P3IA-0001 (PRAIRIE 3IA Institute). Othmane Sebbouh also acknowledges the support of a "Chaire d'excellence de l'IDEX Paris Saclay".

{
\bibliographystyle{apalike}
\bibliography{SGD_Structured}
}

\appendix 
\onecolumn
\aistatstitle{Supplementary Material  \\ SGD for Structured Nonconvex Functions: \\Learning Rates, Minibatching and Interpolation}

The Supplementary Material is organized as follows: In Section~\ref{AppendixTechincalLemmas}, we give some lemmas and consequences of smoothness. In Section~\ref{ProofsExpectedResidual} we present the proofs of the proposition, lemma and theorem related to the Expected Residual condition as presented in Section~\ref{SectionESER} of the main paper. In Section~\ref{AppendixProofs} we present the proofs of the main theorems. In Section~\ref{AppendixTheoryAdditional} we provide additional convergence results under the strongly quasar-convex assumption (Section~\ref{sec:strongquasar}), the Expected Smoothness assumption (Section~\ref{sec:convexpsmooth}) and a  minibatch analysis that does not rely on the interpolation condition (Section~\ref{secapp:minibatch}).

\tableofcontents
\newpage

\section{Technical Lemmas on Smoothness}
\label{AppendixTechincalLemmas}
Here we give some lemmas and consequences of smoothness. 

 For all  of our analysis we do not need that the $f_i$ functions be smooth in all directions. Rather, we just need them to be smooth along  the  $x^*$--direction, as we define next.
 \begin{definition}\label{def:fisgdstarsmooth}
 We say that $f: \R^d \mapsto \R$ is  $L$--smooth function along the $x^*$--direction if there exists $x^*$ such that
\begin{align} \label{eq:smoothnessfuncstar}
 f(z) -f(x) &\leq  \dotprod{\nabla f(x), z-x} +\frac{L}{2}\norm{z-x}^2, \quad  \forall x\in\R^d,
\end{align}
where
\[z = x - \frac{1}{L} (\nabla f(x) - \nabla f(x^*)).\]
By inserting $z$ into~\eqref{eq:smoothnessfuncstar} we can equivalently write~\eqref{eq:smoothnessfuncstar} as
\begin{align} \label{eq:smoothnessfuncstar2}
f\big(x- (1/L) (\nabla f(x) - \nabla f(x^*))\big) \leq f(x) -\frac{1}{2L} \norm{\nabla f(x)}^2 +\frac{1}{2L} \norm{\nabla f(x^*)}^2.
\end{align}
 \end{definition}

\begin{lemma}\label{lem:smoothsubopt} Let $f: \R^d \mapsto \R$ be  differentiable 
and suppose $f$ has a minimizer $x^* \in \R^d.$ Furthermore, let $f$ be
 $L$--smooth function along the $x^*$--direction according to Definition~\ref{def:fisgdstarsmooth}.
It follows that
\begin{equation}\label{eq:smoothsubopt}
\norm{\nabla f(x) }^2 \leq 2 L (f(x) - f(x^*)).
\end{equation}
\end{lemma}

\begin{proof}
Since $x^*$ is a minimizer of $f$ we have that $\nabla f(x^*) =0.$ Furthermore, 
since $f$ is $L$--smooth function along the $x^*$--direction we have by re-arranging~\eqref{eq:smoothnessfuncstar2} that
\[  f(x^*) - f(x) \leq f\big(x- (1/L) \nabla f(x)\big) -f(x) \overset{\eqref{eq:smoothnessfuncstar2}}{\leq} -\frac{1}{2L} \norm{\nabla f(x)}^2.  \]
Re-arranging the above gives~\eqref{eq:smoothsubopt}.
\end{proof}

Now we provide a lemma that will then be used to establish the simplest and most minimalistic assumptions that imply the expected residual (\ref{eq:expresidual}) condition (Assumption~\ref{ass:expresidual}).

\begin{lemma}\label{lem:smoothconvexaroundxst} 
Suppose these exists $x^* \in \R^d$ where
\[x^* \in \argmin \left\{ f(x) \eqdef  \frac{1}{n} \sum_{i=1} f_i(x)\right\},\] 
such that each $f_i$ is convex around $x^*$, that is
\begin{align}
f_i(x^*) -f_i(x) & \leq  \dotprod{\nabla f_i(x^*), x^*-x}, \quad  \forall x\in\R^d,\label{eq:conv}
\end{align}
and each $f_i$ is $L_i$--smooth along the $x^*$--direction according Definition~\ref{def:fisgdstarsmooth}.
It follows for every $i\in\{1,\ldots, n\}$ that
\begin{equation}\label{eq:convandsmooth}
 \norm{\nabla f_i(x) - \nabla f_i(x^*)}^2  \leq 2 L_{i} (f_i(x) -f_i(x^*) - \dotprod{\nabla f_i(x^*), x-x^*}), \quad \forall x\in\R^d.
\end{equation}
 \end{lemma}

\begin{proof}
Fix $i\in \{1,\ldots, n\}$. To prove~\eqref{eq:convandsmooth}, it follows that
\begin{eqnarray}
f_i(x^*) -f_i(x) & = & f_i(x^*) -f_i(z)+f_i(z) - f_i(x)\nonumber\\ 
&\overset{\eqref{eq:conv}+\eqref{eq:smoothnessfuncstar} } \leq &
\dotprod{\nabla f_i(x^*), x^*-z} + \dotprod{\nabla f_i(x), z-x} +\frac{L_i}{2}\norm{z-x}^2,\label{eq:tempsanuin}
\end{eqnarray}
where 
\begin{equation}
z = x - \frac{1}{L_i}(\nabla f_i(x) -\nabla f_i(x^*)).
\end{equation}
Substituting this in $z$ into~\eqref{eq:tempsanuin} gives
\begin{eqnarray}
f_i(x^*) -f_i(x) & = &
\dotprod{\nabla f_i(x^*), x^*-x + \frac{1}{L_i}(\nabla f_i(x) -\nabla f_i(x^*))} - \frac{1}{L_i}\dotprod{\nabla f_i(x), \nabla f_i(x) -\nabla f_i(x^*)} \nonumber \\
& & \quad +\frac{1}{2L_i}\norm{\nabla f_i(x) -\nabla f_i(x^*)}^2 \nonumber \\
& =&\dotprod{\nabla f_i(x^*), x^*-x}   - \frac{1}{L_i}\norm{\nabla f_i(x)-\nabla f_i(x^*)}^2 +\frac{1}{2L_i}\norm{\nabla f_i(x) -\nabla f_i(x^*)}^2 \nonumber \\
&= & \dotprod{\nabla f_i(x^*), x^*-x}   - \frac{1}{2L_i}\norm{\nabla f_i(x)-\nabla f_i(x^*)}^2.\nonumber
\end{eqnarray}
\end{proof} 

Now we present a corollary of the previous lemma for over-parametrized functions
We now develop an immediate consequence of each $f_i$ being convex around $x^*$ and smooth along the   $x^*$-- direction.
\begin{corollary}\label{lem:smoothconvexaroundxstCorollary} 
Suppose these exists $x^* \in \R^d$ where
\[x^* \in \argmin \left\{ f(x) \eqdef  \frac{1}{n} \sum_{i=1} f_i(x)\right\}.\]
Suppose the interpolated Assumption~\ref{ass:over} holds. Furthermore, suppose that for each $f_i$ there exists $L_i$ such that
\begin{align} \label{eq:smoothnessfuncstarover}
f_i\left(x- \frac{1}{L_i} \nabla f_i(x)\right) \leq f_i(x) -\frac{1}{2L_i} \norm{\nabla f_i(x)}^2.
\end{align}
It follows for every $i\in\{1,\ldots, n\}$ that
\begin{equation}\label{eq:convandsmoothover}
 \norm{\nabla f_i(x) - \nabla f_i(x^*)}^2  \leq 2 L_{i} (f_i(x) -f_i(x^*)). \quad \forall x\in\R^d.
\end{equation}
 \end{corollary}
 \begin{proof}
 Note that for interpolated functions we have that each $f_i$ is convex around $x^*$. Furthermore,  since each $\nabla f_i(x^*) =0$ we have that~\eqref{eq:smoothnessfuncstar2} holds, and thus $f_i$ is smooth in the $x^*$--direction according  to Definition~\ref{def:fisgdstarsmooth}. Finally all the conditions of Lemma~\ref{lem:smoothconvexaroundxst} holds, and thus so does~\eqref{eq:convandsmoothover} holds.
 \end{proof}

\section{Proofs of results on Expected Residual}
\label{ProofsExpectedResidual}
\subsection{Proof of Lemma~\ref{lem:varbndrho}}
\begin{proof}
Using
\begin{align*}
\norm{g(x) - \nabla f(x) }^2  & \leq 2\norm{g(x) -g(x^*)- \nabla f(x) }^2  +2\norm{g(x^*)}^2,
\end{align*}
and taking expectation  together with~\eqref{eq:expresidual} and $\nabla f(x^*) =0$  gives
\begin{align*}
\E{\norm{g(x) - \nabla f(x) }^2}  & \leq 4\rho (f(x) -f(x^*))  +2\EE{\cD}{\norm{g(x^*)}^2}.
\end{align*}
Taking the supremum over $x^* \in \cX^*$ and using
~Assumption~\ref{ass:grad-noise}  and that
$\E{\norm{X-\E{X}}^2} = \E{\norm{X}^2} - \norm{\E{X}}^2$ with $X =g(x) $ gives~\eqref{eq:varbndrho2}.
\end{proof}

\subsection{Proof of Proposition~\ref{prop:bniceconst} and its expansion to all samplings.}
\label{sec:lemmaexpres}

In this section we  give an expanded version of Proposition~\ref{prop:bniceconst}  that also gives bounds for the \emph{Expected Smoothness assumption} (ES), a closely related assumption to the Expected Residual condition.

\begin{assumption}[Expected smoothness]\label{ass:expsmooth} 
We say that the stochastic gradient $g$ satisfy the expected smoothness assumption if for all $x \in \mathbb{R}^d$, there exists $\cL =\cL(g) >0$ such that
\begin{equation}
\label{eq:expsmooth}\tag{ES}
\EE{\cD}{\norm{g(x) -g(x^*)  }^2} \leq 2\cL\left(f(x)-f(x^*) \right).
\end{equation}
We use  $g \in \text{ES}(\cL)$ as shorthand for expected smoothness.
\end{assumption}

Here we show that a sufficient condition for the expected smoothness and the expected residual conditions~\ref{ass:expsmooth} and~\ref{ass:expresidual} to hold if that each $f_i$ is convex around $x^*$ and smooth. Furthermore, we  give tight bounds on the expected smoothness $\cL$ and the expected residual constant $\rho$ for when $v$ is an independent sampling and, in particular, a $b$--minibatch sampling.

In the main text our minibatch results are stated only for $b$--minibatching. But they actually hold for a large family of sampling that we refer to as the \emph{independent samplings.}
\begin{definition}[Independent sampling]\label{def:indep}
Let $S \subset \{1,\ldots, n\}$ be a random set and let 
  let $v = \sum_{i\in S} \frac{1}{p_i} e_i$ which is a sampling vector. 
Suppose there exists a constant $c_2>0$ such that
\begin{equation}\label{eq:indep}
 \frac{\Prob{i,j \in S}}{\Prob{i \in S} \Prob{j \in S}} = c_2, \quad \forall i,j \in \{1,\ldots, n\}, \; i\neq j. 
\end{equation}
\end{definition}
In~\cite{gower2019sgd} it was proven that  an independent sampling vector is indeed a valid sampling vector. For completeness we also give the proof  in
In Lemma~\ref{lem:vpisample}.  Furthermore, all the samplings presented in~\cite{gower2019sgd} are examples of an independent sampling vector.
In particular the minibatch sampling in Definition~\ref{def:minibatch} is also an independent sampling.  Finally, note that~\eqref{eq:indep} does not imply that $i \in S$ and $j \in S$ are independent events unless $c_2 =1.$ Indeed, for $b$--minibatch sampling we have that $\Prob{i \in S} = \frac{b}{n} = \Prob{j \in S}$ and $\Prob{i,j \in S} = \frac{b}{n}\frac{b-1}{n-1}$ and thus they are not independent events yet satisfy~\eqref{eq:indep} with $c_2 = \frac{n}{b}\frac{b-1}{n-1}.$

The following Proposition is based on the proof of  Proposition 3.8 in~\cite{gower2019sgd} with the exception that now we show that only convexity around $x^*$ is required for the proof to follow, as opposed to assuming convexity everywhere.

\begin{proposition}  \label{prop:master_lemma}
 Let $f$ be a finite sum problem $f = \frac{1}{n} \sum_{i=1}^n f_i$. Let $f_i$ be $L_i$--smooth  and convex around $x^*$ according to~\eqref{def:fisgdstarsmooth} and~\eqref{eq:conv}, respectively.
It follows that 
\begin{enumerate}
\item  If  $v$ is a sampling vector  then the expected smoothness and expected residual conditions hold $g \in \text{ES}(\cL)$ and $g \in \text{ER}(\rho)$ with $\cL = \max_v \frac{1}{n}\sum_{i=1}^n L_i v_i =\rho.$ 
\item  If  $v$ is an independent sampling vector according to Definition~\ref{def:indep} then we have that
    \begin{align}
        \cL & = c_2 L +\max_{i=1,\ldots, n}\frac{L_i}{np_i}\left(1 -p_ic_2 \right). \label{eq:expsmoothest}\\
       \rho & =   \frac{\lambda_{\max}(\E{(v - \mathbf{1}) (v - \mathbf{1})^\top})}{n}L_{\max} \label{eq:expresidualest}
\end{align}
\item  If $v$ is the $b$--minibatch sampling with replacement then
\begin{align}
\label{eq:sigmaminisupp} 
\sigma^2 &=\frac{1}{b} \frac{n-b}{n-1} \sigma_1^2\\
\label{eq:rhominisupp} 
\rho &= \frac{1}{b} \frac{n-b}{n-1}L_{\max}\\
\label{eq:cLminisupp} 
\cL & = \frac{n}{b}\frac{b-1}{n-1}L+\frac{1}{b}\frac{n-b}{n-1} L_{\max}.
\end{align}
 
\end{enumerate}

\end{proposition}
\begin{proof}
\begin{enumerate}
\item 
Assume that $v$ is any sampling vector. Since $f_i$ is $L_i$--smooth and convex around $x^*$ we have that by multiplying each side of 
\begin{align*}
 f_i(z) -f_i(x) &\leq  \dotprod{\nabla f_i(x), z-x} +\frac{L_i}{2}\norm{z-x}^2\\
 f_i(x^*) -f_i(x) & \leq  \dotprod{\nabla f_i(x^*), x^*-x},
\end{align*}
 by $v_i/n$ and summing up over $i=1,\ldots ,n$ bearing in mind that $v_i \geq 0$ we have that
 \begin{align*}
 f_v(z) -f_v(x) &\leq  \dotprod{\nabla f_v(x), z-x} +\frac{\frac{1}{n}\sum_{i=1}^n v_iL_i}{2}\norm{z-x}^2\\
 f_v(x^*) -f_v(x) & \leq  \dotprod{\nabla f_v(x^*), x^*-x}.
\end{align*}
Consequently $f_v$ is convex and $x^*$ and is $L_v$--smooth where $L_v \eqdef \frac{1}{n}\sum_{i=1}^n  v_i L_i$.
Applying Lemma~\ref{lem:smoothconvexaroundxst} we thus have that
\begin{equation}\label{eq:convandsmoothv}
 \norm{\nabla f_v(x) - \nabla f_v(x^*)}^2  \leq L_{v} (f_v(x) -f_v(x^*) - \dotprod{\nabla f_v(x^*), x-x^*}), \quad \forall x\in\R^d.
\end{equation}
Taking expectation gives
\begin{align*}
 \E{\norm{\nabla f_v(x) - \nabla f_v(x^*)}^2} & \leq \E{ L_{v} (f_v(x) -f_v(x^*) - \dotprod{\nabla f_v(x^*), x-x^*})} \\
 & \leq \max_v L_v \E{ (f_v(x) -f_v(x^*) - \dotprod{\nabla f_v(x^*), x-x^*})} \\
 & = \max_v L_v  (f(x) -f(x^*)).
\end{align*}
This proves that the expected smoothness assumption holds with $\cL =\max_v L_v. $ Consequently by Theorem~\ref{theo:hierarchy} we have that the expected residual condition holds with $\rho = \cL.$

\item
Assume that $v_i$ is an independent sampling. 
First we prove~\eqref{eq:expsmoothest}. 

Since $f_i$ is $L_i$--smooth and convex around $x^*$ we have that $f$ is $L$--smooth and convex around $x^*$ and
by Lemma~\ref{lem:smoothconvexaroundxst}
\begin{align}
\label{eq:procL1}
\|\nabla f_i(x) - \nabla f_i(x^*)\|^2  \leq 2L_i(f_i(x) - f_i(x^*) - \langle \nabla f_i(x^*), x-x^*\rangle )\\
\label{eq:procL2}
\|\nabla f(x) - \nabla f(x^*)\|^2  \leq 2L(f(x) - f(x^*) - \langle \nabla f(x^*), x-x^*\rangle ). 
\end{align}

Noticing that 
\begin{eqnarray*}
\|\nabla f_v(x) - \nabla f_v(x^*)\|^2 &=& \frac{1}{n^2} \left \|\sum_{i\in S}\frac{1}{p_i}(\nabla f_i(x) - \nabla f_i(x^*)) \right \|^2 \\
&=& \sum_{i,j\in S} \left\langle \frac{1}{np_i}(\nabla f_i(x) - \nabla f_i(x^*)), \frac{1}{np_j}(\nabla f_j(x) - \nabla f_j(x^*)) \right\rangle,
\end{eqnarray*}
we have 
\begin{eqnarray*}
\mathbb{E}[\|\nabla f_v(x) - \nabla f_v(x^*)\|^2] &=& \sum_C p_C  \sum_{i,j\in C} \left\langle \frac{1}{np_i}(\nabla f_i(x) - \nabla f_i(x^*)), \frac{1}{np_j}(\nabla f_j(x) - \nabla f_j(x^*)) \right\rangle \\ 
&=& \sum_{i, j=1}^n \sum_{C: i,j\in C }p_C  \left\langle \frac{1}{np_i}(\nabla f_i(x) - \nabla f_i(x^*)), \frac{1}{np_j}(\nabla f_j(x) - \nabla f_j(x^*)) \right\rangle \\ 
&=& \sum_{i, j=1}^n \frac{\Prob{i,j \in S}}{p_ip_j} \left\langle \frac{1}{n}(\nabla f_i(x) - \nabla f_i(x^*)), \frac{1}{n}(\nabla f_j(x) - \nabla f_j(x^*)) \right\rangle,
\end{eqnarray*}
where we used a double counting argument in the 2nd equality.
Now since $\Prob{i,j \in S}/(p_ip_j) = c_2$ for $i \neq j.$ Recalling that $\Prob{i,i \in S}=p_i$ we have from the above that
\begin{eqnarray*}
\mathbb{E}[\|\nabla f_v(x) - \nabla f_v(x^*)\|^2] &=& 
\sum_{i \neq j} c_2 \left\langle \frac{1}{n}(\nabla f_i(x) - \nabla f_i(x^*)), \frac{1}{n}(\nabla f_j(x) - \nabla f_j(x^*)) \right\rangle  \\
& &\qquad + \sum_{i=1}^n\frac{1}{n^2} \frac{1}{p_i} \norm{\nabla f_i(x) - \nabla f_i(x^*))}^2 \\
&= &  \sum_{i,j=1}^n c_2 \left\langle \frac{1}{n}(\nabla f_i(x) - \nabla f_i(x^*)), \frac{1}{n}(\nabla f_j(x) - \nabla f_j(x^*)) \right\rangle \\
& &+ \sum_{i=1}^n\frac{1}{n^2} \frac{1}{p_i}\left(1 -p_ic_2 \right) \norm{\nabla f_i(x) - \nabla f_i(x^*))}^2\\
& \overset{\eqref{eq:procL1}}{\leq} &  c_2 \norm{\nabla f(x) - \nabla f(x^*)}^2 \\
& &+ 2 \sum_{i=1}^n\frac{1}{n^2} \frac{L_i}{p_i}\left(1 -p_ic_2 \right) (f_i(x) - f_i(x^*) - \langle \nabla f_i(x^*), x-x^*\rangle ) \\
& \overset{\eqref{eq:procL2}}{\leq} & 2\left(c_2 L +\max_{i=1,\ldots, n}\frac{L_i}{np_i}\left(1 -p_ic_2 \right)  \right) (f(x) - f(x^*) - \langle \nabla f(x^*), x-x^*\rangle ).
\end{eqnarray*}
Comparing the above to the definition of expected smoothness~\eqref{eq:expsmooth} we have that
\begin{equation} \label{eq:CLinterpolc2}
\cL \quad \leq \quad  c_2 L +\max_{i=1,\ldots, n}\frac{L_i}{np_i}\left(1 -p_ic_2 \right).
\end{equation}

Now we will prove that
\begin{equation}\label{eq:tempsnioenrs}
    \E{\norm{\nabla f_v (w) -\nabla f_v (x^*) - ( \nabla f(w)-\nabla f(x^*)) }^2} \leq 2\rho \left(f(w)-f(x^*) \right),
\end{equation}
holds with the constant given in~\eqref{eq:expresidualest}.
First we expand the squared norm on the left hand side of~\eqref{eq:tempsnioenrs}. Define $DF(w) = [\nabla f_1(w), \ldots, \nabla f_n(w)] \in \R^{d \times n}$ as the Jacobian of $F(w) \overset{def}{=} [f_1(w),\dots,f_n(w)]$. We denote $\mR \eqdef \left(DF (w) - DF (x^*)\right)$. It follows that
\begin{align*}
    C &\eqdef \norm{\nabla f_v (w) -\nabla f_v (x^*) - (\nabla f(w)-\nabla f(x^*)) }^2\\
    &= \frac{1}{n^2}\norm{\left(DF (w) - DF (x^*)\right) ( v - \mathbf{1})}^2\\
    &= \frac{1}{n^2} \langle \mR ( v - \mathbf{1}), \mR ( v - \mathbf{1}) \rangle_{\R^d}\\
    &=\frac{1}{n^2} \trace{(v - \mathbf{1})^\top \mR^\top \mR (v - \mathbf{1})}\\
    &= \frac{1}{n^2} \trace{\mR^\top \mR (v - \mathbf{1}) (v - \mathbf{1})^\top}.\\
\end{align*}
Let $\Var{v} \eqdef \E{(v - \mathbf{1}) (v - \mathbf{1})^\top}.$
Taking expectation,
\begin{eqnarray}
\E{C} &=& \frac{1}{n^2} \trace{\mR^\top \mR \Var{v}} \nonumber \\
& \leq & \frac{1}{n^2} \trace{\mR^\top \mR} \lambda_{\max}(\Var{v}). \label{eq:trace-eig}
\end{eqnarray}

Moreover, since the $f_i$'s are convex  around $x^*$ and $L_i$-smooth, it follows from~\eqref{eq:convandsmooth} that
\begin{eqnarray}
\trace{\mR^\top \mR} &=& \sum_{i=1}^n \norm{\nabla f_i(w) - \nabla f_i(x^*)}^2 \nonumber \\
&\leq& 2 \sum_{i=1}^n L_i (f_i(w) - f_i(x^*) - \langle \nabla f_i(x^*), w - x^* \rangle) \nonumber \\
&\leq& 2nL_{\max}(f(w) - f(x^*)). \label{eq:traceRtT}
\end{eqnarray}
Therefore,
\begin{eqnarray}
\E{C}\overset{\eqref{eq:trace-eig} + \eqref{eq:traceRtT}}{\leq}  2\frac{\lambda_{\max}(\Var{v})}{n}L_{\max} (f(w) - f(x^*)).
\end{eqnarray}
Which means
\begin{equation} \label{eq:residual_result}
    \rho = \frac{\lambda_{\max}(\Var{v})}{n}L_{\max}.
\end{equation}
\item Finally, if $v$ is a $b$--minibatch sampling, the specialized expressions for $\cL$ in~\eqref{eq:cLminisupp} follows by observing that $\Prob{i\in S} = p_i = \frac{b}{n}$, $\Prob{i,j\in S} = \frac{b}{n}\frac{b-1}{n-1} $ and consequently $c_2 =\frac{n}{b}\frac{b-1}{n-1}. $ The specialized expressions for 
 $\sigma$ and $\rho$ in~\eqref{eq:sigmaminisupp} and~\eqref{eq:rhominisupp} follow  
from Proposition 3.8~\cite{gower2019sgd} and  Lemma  F.3  in~\cite{sebbouh2019towards}, respectively.
\end{enumerate}
\end{proof}
     
\subsection{Proof of Theorem~\ref{theo:hierarchy}}
\label{sec:hierarchy}
First we include the formal definition of each of these assumptions named in Theorem~\ref{theo:hierarchy}. Let $g(x) = \nabla f_i(x)$ denote the stochastic gradient.
 The results in this section carry over verbatim by using $g(x) = \nabla f_v(x)$ and $f_i =f_v$ instead, where $v$ is a sampling vector. But since the sampling only affects the constants in each of the forthcoming assumptions, and here we are only interested in a hierarchy between assumptions, we omit the proof for a general sampling vector.  

First we repeat the definitions of
$ES$, $WGC$ and $SGC$ from Assumption 2~\cite{Khaled-nonconvex-2020}, Assumption 2.1 in~\cite{gower2019sgd}, Eq (7)  and Eq (2) in~\cite{vaswani2018fast}, respectively. 

\paragraph{SGC: Strong Growth Condition.} We say that $SGC$ holds with $\rho_{SGC}>0$ if
\begin{equation}\label{eq:SGC}
\E{\norm{g(x)}^2} \leq \rho_{SGC}\norm{\nabla f(x)}^2.
\end{equation}
\paragraph{WGC: Weak Growth Condition.} We say that $WGC$ holds with $\rho_{WGC}>0$ if
\begin{equation}\label{eq:WGC}
\E{\norm{g(x)}^2} \leq 2\rho_{WGC}(f(x) -f(x^*)).
\end{equation}
\paragraph{ES: Expected Smoothness.} We say that $ES$ holds with $\cL>0$ if
\begin{equation}\label{eq:ES}
\E{\norm{g(x) -g(x^*)}^2} \leq 2\cL(f(x) -f(x^*)).
\end{equation}
\paragraph{ER: Expected Residual.} We say that $ER$ holds with $\rho>0$ if
\begin{equation}\label{eq:ER}
\E{\norm{g(x) -g(x^*) - ( \nabla f(x)-\nabla f(x^*)) }^2} \leq 2\rho\left(f(x)-f(x^*) \right).
\end{equation}
In addition we will use
\paragraph{$x^*$--convex.} We say that $x^*$--convex holds if
\begin{equation}\label{eq:xstarconvex}
f_i(x^*) -f_i(x)  \leq  \dotprod{\nabla f_i(x^*), x^*-x}, \quad \mbox{for }i=1,\ldots, n.
\end{equation}
\paragraph{$L_i$--smoothness.} We say that $L_i$--smoothness holds for $L_i>0$ if
\begin{equation}\label{eq:Li}
 f_i(z) -f_i(x) \leq  \dotprod{\nabla f_i(x), z-x} +\frac{L_i}{2}\norm{z-x}^2, \quad  \forall x,z\in\R^d, \; i=1,\ldots, n.
\end{equation}
\paragraph{Interpolated.} We say that the interpolation condition holds at $x^*$ if
\begin{equation}\label{eq:LiInterpolated}
 f_i(x^*) \leq f_i(x), \quad \mbox{for }i=1,\ldots, n,  \mbox{and for every } x \in \R^d.
\end{equation}

An important assumption created recently~\cite{Khaled-nonconvex-2020} is the following $ABC$--assumption
\paragraph{ABC.} We say that $ABC$ holds with $A,B,C>0$ if
\begin{equation}\label{eq:ABC}
\E{\norm{g(x)}^2} \leq 2A(f(x) -f(x^*) +B\norm{\nabla f(x)}^2 + C.
\end{equation}
The $ABC$ condition~\eqref{eq:ABC} includes all previous assumptions SGC, WGC, ES and ER as a special case by choosing the three parameters $A,B$ and $C$ appropriately. In this sense, it is a family of assumptions. See~\cite{Khaled-nonconvex-2020} for more details on this assumption and how it linked to all the other assumptions.

Now we repeat the statement of Theorem~\ref{theo:hierarchy} for convenience.

\begin{theorem} \label{prop:hierarchyapp}
The following hierarchy holds
\[
{
\begin{array}
[c]{cccccccccccc}
\boxed{SGC+L\mbox{--smooth}}& \Longrightarrow& \boxed{WGC }& \Longrightarrow & \boxed{ES} & \Longrightarrow & \boxed{\colorbox{blue!20}{ER}}  &  \Longrightarrow & ABC \\
& &  &  &   \Uparrow & & &&\\
& & \boxed{L_i+\mbox{Interpolated}}& \Longrightarrow &  \boxed{ L_i+x^*\mbox{--convex} }& &  
\end{array}
}
\]
In addition we have that $ES(\cL)+ PL(\mu) \Rightarrow ER(\cL-\mu)$ and $ER\nRightarrow ES.$
\end{theorem}
\begin{proof} We first prove the top row of implications.\\\
\noindent {\bf 1.} {$SGC +L\mbox{--smooth} \implies WGC$.} Using Lemma~\ref{lem:smoothsubopt} and~\eqref{eq:SGC} we have that
\begin{eqnarray*}
\E{\norm{g(x)}^2} & \leq & \rho_{SGC}\norm{\nabla f(x)}^2 \\
&\overset{\eqref{eq:smoothsubopt}}{ \leq} & 2L \rho_{SGC} (f(x) -f(x^*)).
\end{eqnarray*}
   Thus~\eqref{eq:WGC} holds with $\rho_{WGC} =  2L \rho_{SGC}.$
   
\noindent {\bf 2.} {$WGC \implies ES$.} 

Plugging in $x = x^*$ in WGC~\eqref{eq:WGC} gives $g(x^*) = 0$ almost surely. Since  $g(x^*) = 0$
we have that ~\eqref{eq:WGC} gives~\eqref{eq:ES}.  

\noindent {\bf 3.}{$ES \implies ER$.} 
Expanding the squares of the left hand side of~\eqref{eq:ER} gives
\begin{align*}
\norm{g(x) -g(x^*) - ( \nabla f(x)-\nabla f(x^*)) }^2 & =
\norm{g(x) -g(x^*)}^2 + \norm{\nabla f(x)-\nabla f(x^*)}^2\\
& \quad  -2 \dotprod{g(x) -g(x^*),\nabla f(x)-\nabla f(x^*)} .
\end{align*}
Now assuming that $ES$~\eqref{eq:ES} holds, taking expectation and using that $\E{g(x) } = \nabla f(x)$ we have that
\begin{align*}
\E{\norm{g(x) -g(x^*) - ( \nabla f(x)-\nabla f(x^*)) }^2} & =
\E{\norm{g(x) -g(x^*)}^2} - \norm{\nabla f(x)-\nabla f(x^*)}^2\\
& \leq \E{\norm{g(x) -g(x^*)}^2} \\
& \leq 2\cL (f(x) - f^*). 
\end{align*}
In addition, if the PL condition holds, then we can upper bound $- \norm{\nabla f(x)-\nabla f(x^*)}^2 \leq -2\mu (f(x) - f^*)$ which combined with the above gives
\[ \E{\norm{g(x) -g(x^*) - ( \nabla f(x)-\nabla f(x^*)) }^2} \leq 2(\cL-\mu) (f(x) - f^*).\]
Thus $ER$ holds with $\rho = \cL -\mu.$

Now we prove the remaining implications.\\\

\noindent {\bf 4.} {$L_i+$Interpolated  $\implies L_i+x^*\mbox{--convex}$.} A direct consequence of the interpolation assumption~\eqref{ass:over} is that $\nabla f_i(x^*) =0$ and $f_i(x^*) \leq f_i(x).$ Consequently $f_i(x^*) \leq f_i(x) + \dotprod{\nabla f_i(x^*),x-x^*}$.

\noindent {\bf 5.} { $L_i+x^*\mbox{--convex} \implies ES$.} Follows from Proposition~\ref{prop:master_lemma}.

\noindent {\bf 6.} { $ER \nRightarrow ES$.}
Since when $v$ encodes the full batch sampling where $g (x) = \nabla f(x)$, the expected residual condition always holds for any $\rho >0$ since the left hand side of~\eqref{eq:expresidual} is zero and $0 \leq \rho (f(x)-f^*).$ On the other hand, in the full batch case the expected smoothness assumption is equivalent to claiming that $f$ is $L$--smooth, and clearly there exist differentiable functions that have gradients that are not Lipschitz. For instance $f(x) = x^4.$

Finally\\
\noindent {\bf 7.} { $ER \Rightarrow ABC$.} If the ER condition holds, by Lemma~\ref{lem:varbndrho}  we have that~\eqref{eq:varbndrho2} holds, which fits the format of the ABC assumption~\eqref{eq:ABC} where $A = 2 \rho$, $B = 1$ and $C = 2\sigma^2.$
\end{proof}

\section{Proofs of Main Convergence Analysis Results}
\label{AppendixProofs}

\subsection{Proof of Theorem~\ref{theo:master-quasar-convex-res} }

First we need the following lemma.
\begin{lemma}
Assume $g \in ER(\rho)$. Then for all $x \in \R^d$,
\begin{eqnarray}\label{eq:varbndrho3}
\EE{\cD}{\norm{g(x)}^2} \leq 2(2\rho + L)(f(x) - f(x^*) +2 \sigma^2.
\end{eqnarray}
\end{lemma}
\begin{proof}
Since $f$ is $L-$smooth, we have $\norm{\nabla f(x)}^2 \leq 2L(f(x) - f(x^*)).$ Using this inequality together with \eqref{eq:varbndrho2} gives \eqref{eq:varbndrho3}.
\end{proof}

\begin{proof}
We have:
\begin{align*}
\norm{x^{k+1} - x^*}^2 &= \norm{x^k - x^*}^2 - 2\gamma_k\langle g(x^k), x^k - x^*\rangle + \gamma_k^2\norm{g(x^k)}^2
\end{align*}
Hence, taking expectation conditioned on $x_k$, we have:
\begin{eqnarray*}
\EE{\cD}{\norm{x^{k+1} - x^*}^2} &=& \norm{x^k - x^*}^2 - 2\gamma_k\langle \nabla f(x^k), x^k - x^*\rangle + \gamma_k^2\EE{\cD}{\norm{\nabla f_{v_k}(x_k)}^2}\\
&\overset{\eqref{eq:quasar-convex}+\eqref{eq:varbndrho3}}{\leq}& \norm{x^k - x^*} - 2\gamma_k(\zeta - \gamma_k(2\rho + L))(f(x^k) - f^*)) + 2\gamma_k^2\sigma^2.
\end{eqnarray*}
Rearranging and taking expectation, we have
\begin{align*}
2\gamma_k(\zeta - \gamma_k(2\rho + L))\E{f(x^k) - f^*} \leq \E{\norm{x^k - x^*}^2} - \E{\norm{x^{k+1} - x^*}^2} + 2\gamma_k^2\sigma^2.
\end{align*}
Summing over $k =0,\ldots, t-1$ and using telescopic cancellation gives
\begin{align*}
2\sum_{k=0}^{t-1}\gamma_k(\zeta - \gamma_k(2\rho + L))\E{f(x_k) - f^*} \leq \norm{x^0 - x^*}^2 - \E{\norm{x^{k} - x^*}^2} + 2\sigma^2\sum_{k=0}^{t-1}\gamma_k^2.
\end{align*}
Since $\E{\norm{x^{k} - x^*}^2} \geq 0$ and $(\zeta - \gamma_k(2\rho + L))\geq 0$, dividing both sides by $2\sum_{i=0}^{t-1}\gamma_i(\zeta - \gamma_i(2\rho + L))$ gives:
\begin{align*}
\sum_{k=0}^{t-1}\E{\frac{\gamma_k(\zeta - \gamma_k(2\rho + L)}{\sum_{i=0}^{t-1}\gamma_i(\zeta - \gamma_i(2\rho + L))}(f(x^k) - f^*)} \leq \frac{\norm{x^0 - x^*}^2}{2\sum_{i=0}^{t-1}\gamma_i(\zeta - \gamma_i(2\rho + L))} + \frac{\sigma^2\sum_{k=0}^{t-1}\gamma_k^2}{\sum_{i=0}^{t-1}\gamma_i(\zeta - \gamma_i(2\rho + L))}.
\end{align*}
Thus,
\begin{align*}
\min_{k=0,\dots,t-1}\E{f(x^k) - f(x^*)} \leq \frac{\norm{x^0 - x^*}^2}{2\sum_{i=0}^{t-1}\gamma_i(\zeta - \gamma_i(2\rho + L))} + \frac{\sigma^2\sum_{k=0}^{t-1}\gamma_k^2}{\sum_{i=0}^{t-1}\gamma_i(\zeta - \gamma_i(2\rho + L))}.
\end{align*}
For the different choices of step sizes:
\begin{enumerate}
\item If $\forall k \in \N, \; \gamma_k = \frac{1}{2} \frac{\zeta}{(2\rho + L)}$, then it suffices to replace $\gamma_k = \gamma$ in \eqref{eq:master-quasar-convex-res}.
\item Suppose algorithm \eqref{eq:sgdstep} is run for $T$ iterations. Let $\forall k=0,\dots,T-1, \; \gamma_k = \frac{\gamma}{\sqrt{T}}$ with $\gamma \leq \frac{\zeta}{2(2\rho + L)}$. Notice that since $\gamma \leq \frac{\zeta}{2(2\rho + L)}$, we have $\zeta - \gamma(2\rho + L) \leq \frac{1}{2}$. Then it suffices to replace $\gamma_k = \frac{\gamma}{\sqrt{T}}$ in \eqref{eq:master-quasar-convex-res}.

\item Let $\forall k \in \N, \; \gamma_k = \frac{\gamma}{\sqrt{k+1}}$ with $\gamma \leq \frac{\zeta}{2\rho + L}$. Note that that since $\gamma_t = \frac{\gamma}{\sqrt{t+1}}$ and using the integral bound, we have that
\begin{equation}
\sum_{t=0}^{k-1}\gamma_t^2  = \gamma^2 \sum_{t=0}^{k-1} \frac{1}{t+1} \; \leq \; \gamma^2\left(\log(k) + 1 \right). \label{eq:ks94oo8s84-res}
\end{equation}
Furthermore using the integral bound again we have that
\begin{eqnarray}
\sum_{t=0}^{k-1}\gamma_t &\geq& 2\gamma\left(\sqrt{k} -1\right)
\label{eq:ia37ha37ha3-res}.
\end{eqnarray}
Now using~\eqref{eq:ks94oo8s84-res} and~\eqref{eq:ia37ha37ha3-res} we have that
\begin{eqnarray*}
\sum_{i=0}^{k-1}\gamma_i(\zeta-\gamma_i(2\rho + L)) &= &
\zeta\sum_{i=0}^{k-1}\gamma_i -(2\rho + L) \sum_{i=0}^{k-1}\gamma_i^2
\\
&\geq & 2\gamma\left( \zeta(\sqrt{k} -1) - \gamma (\rho + \frac{L}{2})\left(\log(k) + 1 \right) \right).
\end{eqnarray*}
It remains to replace bound the sums in \eqref{eq:master-quasar-convex-res} by the values we have computed.
\end{enumerate}
\end{proof}

\subsection{Proof of Corollary~\ref{cor:quasar-minib} }
\begin{proof}
The interpolated assumption~\ref{ass:over} implies that $\nabla f_i(x^*) = g(x^*) = 0$ and thus $\sigma =0.$ Furthermore from~\eqref{eq:bniceconst} we have that the \ref{eq:expresidual} condition holds with $\rho = L_{\max} \frac{n-b}{(n-1)b}$. Combining these two observations with \eqref{eq:master-quasar-const} gives~\eqref{eq:SGDforStarOVerER}. The total complexity~\eqref{eq:totalcomplexStaroverER} follows from computing the iteration complexity via~\eqref{eq:SGDforStarOVerER}  and multiplying it by $b$.

 Finally for the optimal minibatch size,  since~\eqref{eq:totalcomplexStaroverER} is a linear function in  $b$, the minimum depends on  the sign of its slope. Taking the derivative in $b$ we have the slope is given by $2\frac{L- 2L_{\max}}{n-1} $. If the slope is negative, we want $b$ to be a large as possible, that is $b=n$. Otherwise if the slope is positive $b=1$  is optimal.
\end{proof}
\subsection{Proof of Lemma~\ref{lem:cLmax}}

Before presenting our proof for Lemma~\ref{lem:cLmaxapp}, we need to present a large family of sampling vectors called the \emph{arbitrary samplings}.
\begin{lemma}[Lemma 3.3 ~\cite{gower2019sgd}]\label{lem:vpisample}
Let $S \subset \{1,\ldots, n\}$ be a random set.
Let $\Prob{i \in S} = p_i$. It follows that $v = \sum_{i\in S} \frac{1}{p_i} e_i$ is a sampling vector. We call $v$ the \emph{arbitrary sampling} vector.
\end{lemma}
An arbitrary sampling is sufficiently flexible as to model almost all samplings and minibatching schemes of interest, see Section 3.2 i~n\cite{gower2019sgd}. For example the $b$--minibatch sampling is a special case where  $p_i = \frac{b}{n}$ and $\Prob{S = B} = \left. 1 \right/ \binom{n}{b}$ for every $B \in \{1,\ldots, n\}$ that has $b$ elements.

Now we prove  Lemma~\ref{lem:cLmax} and some additional results. 
\begin{lemma}\label{lem:cLmaxapp}
Assume interpolation~\ref{ass:over} holds. Let $f_i$ be $L_i$--smooth and let $v$ be a sampling vector as defined in Lemma~\ref{lem:vpisample}. It follows that there exists $\cL_{\max} > 0 $ such that
\begin{equation} \label{eq:cLmaxbndapp}
\frac{1}{2  \cL_{\max}} ( f(x) -f^* ) \leq \E{\frac{(f_v(x)-f_v^*)^2}{ \|\nabla f_v(x)\|^2}}.
\end{equation}
For $B \subset\{1,\ldots, n\}$ 
let $L_B$ be the smoothness constant of $f_B \eqdef\frac{1}{n} \sum_{i\in B} p_i f_i$. It follows that\\
{\bf 1.} If $v$ is an arbitrary sampling vector (Lemma~\ref{lem:vpisample}) then 
$\cL_{\max} = \underset{i=1,\ldots, n}{ \max} \dfrac{p_i }{ \sum_{B: i \in B}\frac{p_B}{ L_B} }.$\\
{\bf 2.} If $v$ is the $b$--minibatch sampling then 
$\cL_{\max} \; =\; \cL_{\max}(b) \;=\; \underset{i=1,\ldots, n}{ \max} \dfrac{  \binom{n-1}{b-1}}{ \sum_{B: i \in B} L_B^{-1}}.$\\
   \end{lemma}
   
\begin{proof}
Since $f_i$ is $L_i$--smooth, we have that $f_v$ is $L_v$--smooth with $L_v \eqdef \frac{1}{n} \sum_{i=1}^n v_i L_i.$  Thus according to Lemma~\ref{lem:smoothconvexaroundxst} we have that
\begin{equation*}
 \| \nabla f_v(x)\|^2 \leq 2 L_v (f_v(x)-f_v^*).
\end{equation*}
Consequently we have that
\begin{equation}\label{eq:tempgradvbnd}
\frac{1}{\| \nabla f_v(x)\|^2} \geq \frac{1}{2L_v (f_v(x)-f_v^*)}.
\end{equation}
Using this we have the following bound
\begin{align}\label{eq:tempasdasda}
\E{\frac{ (f_v(x)-f_v^*)^2}{ \| \nabla f_v(x)\|^2}}  
 \overset{\eqref{eq:tempgradvbnd}}{\geq}\E{\frac{ f_v(x)-f_v^*}{ 2 L_v}}.
\end{align}
Let $S$ be the random set associated to the arbitrary sampling vector $v$. We use $B\subset \{1,\ldots, n\}$ to denote a realization of $S$ and $p_B \eqdef \Prob{B=S}.$ 
  Thus with this notation we have that
\begin{eqnarray}
\E{\frac{ (f_v(x)-f_v^*)^2}{ \| \nabla f_v(x)\|^2}}   & \overset{\eqref{eq:tempasdasda}}{\geq} &
 \sum_{B \subset \{1,\ldots ,n\} } p_B\frac{ f_B(x)-f_B^*}{2L_B}. \label{eq:tempsinejoise}
 \end{eqnarray}
 
  Now let   $p_i \eqdef \Prob{i \in S}.$ Due to the interpolation condition we have that and the definition of $f_B$ we have that
  \[f_B^* = f_B(x^*) = \frac{1}{n} \sum_{i\in B} p_i f_i(x^*) = \frac{1}{n} \sum_{i\in B} p_i f_i^*.  \]
Consequently
 \begin{eqnarray}
\E{\frac{ (f_v(x)-f_v^*)^2}{ \| \nabla f_v(x)\|^2}}   &\overset{\eqref{eq:tempsinejoise}}{=}&\sum_{B \subset \{1,\ldots ,n\} } p_B \sum_{i \in B} \frac{ f_i(x)-f_i^*}{2n L_B p_i} \nonumber \\
&= &\frac{1}{2n}\sum_{i=1,\ldots, n  }  \sum_{B: i \in B}\frac{p_B}{ p_i L_B} (f_i(x)-f_i^*) \nonumber \\
& \geq & \min_{i=1,\ldots, n} \left\{ \sum_{B: i \in B}\frac{p_B}{ p_i L_B} \right\}  \frac{1}{2cn}\sum_{i=1,\ldots, n  }(f_i(x)-f_i^*)\nonumber \\
& = & \frac{1}{2}\min_{i=1,\ldots, n} \left\{ \sum_{B: i \in B}\frac{p_B}{ p_i L_B} \right\}  (f(x) -f^*),
\end{eqnarray}
where in the first equality we used a double counting argument to switch the order of the sum over subsets $B$ and elements $i\in B$.
The main result~\eqref{eq:cLmaxbndapp} now follows by observing that

\[\frac{1}{\min_{i=1,\ldots, n} \left\{ \sum_{B: i \in B}\frac{p_B}{ p_i L_B} \right\} } = \max_{i=1,\ldots, n}\left\{ \frac{p_i }{ \sum_{B: i \in B}\frac{p_B}{ L_B}  } \right\}
 =\cL_{\max}.\]

Finally, for a $b$--minibatch sampling we have that
\[p_i = \frac{b}{n}, \quad p_B = \left. 1\right/ \binom{n}{b} \quad \mbox{and}\quad L_B  \leq  \frac{1}{b}\sum_{j\in B} L_j,\]
which in turn gives
\[\frac{1}{\cL_{\max}} =\min_{i=1,\ldots, n}\sum_{B: i \in B}\frac{n}{ b} \frac{1}{\binom{n}{b}} \frac{1}{L_B}   = \min_{i=1,\ldots, n}\sum_{B: i \in B}\frac{1}{\binom{n-1}{b-1}} \frac{1}{L_B}  .\]
\end{proof}

\subsection{Proof of Theorem~\ref{theo:SPLquasarconvex}}
\begin{proof}
\begin{eqnarray}
\label{noiaks}
\|x^{k+1}-x^*\|^2&=&\|x^k-\gamma_k \nabla f_v(x^k)-x^*\|^2\notag\\
&=&\|x^k-x^*\|^2-2 \gamma_k \langle x^k-x^*, \nabla f_v(x^k) \rangle + \gamma_k^2 \| \nabla f_v(x^k)\|^2\notag\\
&\overset{\eqref{eq:quasar-convex} }{\leq}&\|x^k-x^*\|^2-2 \zeta \gamma_k \left[f_v(x^k)-f_v(x^*)\right] + \gamma_k^2 \| \nabla f_v(x^k)\|^2\notag\\
&\overset{\eqref{SPLRv}}{=}&\|x^k-x^*\|^2-2 \zeta \gamma_k \left[f_v(x^k)-f_v^*\right] + \frac{\gamma_k }{c}\left[f_v(x^k)-f_v^*\right]\notag\\
&=&\|x^k-x^*\|^2-\gamma_k\left(2 \zeta-\frac{1}{c}\right) \left[f_v(x^k)-f_v(x^*)\right].
\end{eqnarray}

By rearranging we have that
\begin{eqnarray}
\label{noiaks2}
\gamma_k\left(2 \zeta-\frac{1}{c}\right) \left[f_v(x^k)-f_v(x^*)\right] &\leq& \|x^k-x^*\|^2 - \|x^{k+1}-x^*\|^2.
\end{eqnarray}
Taking expectation, and 
since $2 \zeta-\frac{1}{c}>0$ we have by Lemma~\ref{lem:cLmax} we have that
\begin{eqnarray}
\frac{2c \zeta-1}{2c^2} \frac{1}{\cL_{\max}} \E{ f(x^k) -f(x^*)}  &\leq & \left(2 \zeta-\frac{1}{c}\right)\E{\frac{ (f_v(x)-f_v^*)^2}{c \| \nabla f_v(x)\|^2}}   \nonumber \\
& \overset{\eqref{SPLRv}}{=}&  \left(2 \zeta-\frac{1}{c}\right)\E{\gamma_k (f_v(x)-f_v^*)} \nonumber \\
 &\overset{\eqref{noiaks2}}{ \leq} & \E{\|x^k-x^*\|^2 }- \E{\|x^{k+1}-x^*\|^2}.\label{noiaks3}
\end{eqnarray}
Summing from $k= 0, \ldots, K-1$ and using telescopic cancellation gives
\begin{eqnarray}
\label{noiaks4}
\frac{2c \zeta-1}{2c^2} \frac{1}{\cL_{\max}}\sum_{k=0}^{K-1}\E{ f(x^k) -f(x^*)}   &\leq& \|x^{0}-x^*\|^2 - \E{\|x^{K}-x^*\|^2}.\notag
\end{eqnarray}
Multiplying through by  $\cL_{\max}\frac{2c^2} {2c \zeta-1}\frac{1}{K} $ gives
\begin{align*}
\min_{i=0,\ldots, K-1} \E{f(x^k) -f(x^*)}  \leq  \frac{1}{K}\sum_{k=0}^{K-1} \E{ f(x^K) -f(x^*)} & \leq \frac{2c^2} {2c \zeta-1}\frac{ \cL_{\max}}{K} \|x^{0}-x^*\|^2 .
\end{align*}
\end{proof}

\subsection{Proof of Theorem~\ref{theo:PLConstant}}

In the following proof, for ease of reference, we repeat the step-size choice here:
\begin{equation}\label{ncaoikaosd}
\gamma\leq \frac{1}{1 +2\rho/\mu} \frac{1}{L}.
\end{equation}
\begin{proof}
By combining the smoothness of function $f$ with the update rule of SGD we obtain:
\begin{eqnarray}
f(x^{k+1})& \leq &  f(x^{k})+ \langle \nabla f(x^k), x^{k+1}-x^k \rangle +\frac{L}{2} \| x^{k+1}-x^k\|^2 \notag\\
&=& f(x^{k})-\gamma\langle \nabla f(x^k), \nabla f_{v^k}(x^k) \rangle +\frac{L \gamma^2}{2} \| \nabla f_{v^k}(x^k)\|^2.
\end{eqnarray}

By taking expectation conditioned on $x^k$ we obtain:
\begin{eqnarray}
\label{eq:sgdsbuopt1sttstep}
\E{f(x^{k+1})\; | \; x^k} & \leq & f(x^{k}) - \gamma \norm{\nabla f(x^k)}^2 + \frac{L\gamma^2}{2} \Exp_\cD{\norm{\nabla f_{v^k}(x^k)}^2}  \notag\\
& \overset{\eqref{eq:varbndrho2}}{\leq} &  f(x^{k}) - \gamma \norm{\nabla f(x^k)}^2 + 2L\gamma^2 \rho ( f(x^k)-f^* ) \notag\\&& +  \frac{L\gamma^2}{2}\|\nabla f (x^k)\|^2   + L\gamma^2 \sigma^2 \nonumber\\
& =& f(x^{k})-\gamma(1 -\frac{L\gamma}{2})\norm{\nabla f(x^k)}^2 \notag\\&& +2L\gamma^2\rho(f(x^k) -f^*)+L\gamma^2\sigma^2  \notag\\ 
&\overset{\eqref{eq:PL}}{ \leq }& f(x^{k})-2\mu\gamma(1 -\frac{L\gamma}{2})[(f(x^k) -f^*)] \notag\\&& +2L\gamma^2\rho(f(x^k) -f^*)+L\gamma^2\sigma^2,
\end{eqnarray} 
where the last inequality holds because $1 -\frac{L\gamma}{2}>0$ since $\gamma\leq \frac{1}{1 +2\rho/\mu} \frac{1}{L} <\frac{1}{L}.$

Taking expectations again and subtracting $f^*$ from both sides yields:
\begin{eqnarray}
\Exp[f(x^{k+1})-f^*] & \leq & \bigg( 1-2\gamma\left(\mu(1 -\frac{L\gamma}{2}) -L\gamma\rho\right) \bigg)\Exp[f(x^k) -f^*] +L\gamma^2\sigma^2. \notag\\ 
& \overset{\eqref{ncaoikaosd}}{\leq} & (1-\gamma \mu) \Exp[f(x^k) -f^*] +L\gamma^2\sigma^2.  \label{eq:subPLconvexstep}
\end{eqnarray}

Recursively applying the above and summing up the resulting geometric series gives:
\begin{eqnarray}
\E{f(x^{k}) -f^*} & \leq & (1-\mu \gamma)^{k}[f(x^0) -f^*]
+ L\gamma^2\sigma^2 \sum_{j=0}^{k-1} (1- \gamma \mu )^j.
\end{eqnarray}
Using $\sum_{i=0}^{k-1} (1-\mu \gamma)^i = \frac{1-(1-\mu \gamma)^{k}}{1-1+\mu \gamma} \leq \frac{1}{\mu \gamma},$
in the above gives~\eqref{eq:functionTheoremExpResidual}.

\textbf{On Iteration Complexity:} 
For ease of reference, we
we repeat the step-size choice for the iteration complexity result 
\begin{equation}\label{eq:gammaSGDPLcompl}
\gamma =\frac{1}{L}\min \left\{ \frac{\mu \epsilon}{2 \sigma^2}, \, \frac{1}{1 +2\rho/\mu}\right\}
\end{equation}

To analyze the iteration complexity, let $\epsilon>0$ and let us divide the right hand side of~\eqref{eq:functionTheoremExpResidual} into two parts and bound each of them separately by $\frac{\epsilon}{2}$. For the right most part we have that
\begin{eqnarray}
\label{naoisdaldka}
\frac{L\gamma\sigma^2}{\mu} & \leq  & 
 \frac{\epsilon}{2} \quad \Rightarrow \quad 
 \gamma \; \leq \; \frac{1}{L} \frac{\mu \epsilon}{2 \sigma^2}.
\end{eqnarray}
The derivation in \eqref{naoisdaldka} gives us the restriction~\eqref{eq:gammaSGDPLcompl} on the step size.

For the other remaining part we have that
 \[(1-\mu \gamma )^k(f(x^0)-f^*) < \frac{\epsilon}{2}.\]
 Taking logarithms and re-arranging the above gives
\begin{equation}
  \log\left(\frac{2 (f(x^0)-f^*)}{  \epsilon }\right) \leq k\log\left(\frac{1}{1 - \gamma \mu }\right).\label{eq:acb378babuhi3}
\end{equation}
Now using that $\log\left(\frac{1}{\rho}\right) \geq 1-\rho,$ for $0<\rho \leq 1$ gives
\[k\geq \frac{1}{\mu \gamma }\log\left(\frac{\epsilon}{2(f(x^0)-f^*)}\right).\]
Thus restricting the step size according to~\eqref{eq:gammaSGDPLcompl} and inserting $\gamma$ into the above gives the result~\eqref{eq:itercomplexPL}.
\end{proof}

\subsubsection{Comparison of Theorem~\ref{theo:PLConstant} to the PL Convergence Results by~\cite{Khaled-nonconvex-2020}}
\label{sec:comparekhaled}
Recently, \cite{Khaled-nonconvex-2020} present a comprehensive theory for the convergence of SGD in the nonconvex setting. They do this by relying on ABC assumption~\eqref{eq:ABCpaper}. \cite{Khaled-nonconvex-2020}  consider the general smooth non-convex setting, with no additional assumption besides the ABC assumption, where they establish new state-of-the-art results. They also consider some additional assumptions such as in Theorem 3 in~\citep{Khaled-nonconvex-2020} where they assume that the PL condition~\eqref{eq:PL} holds. Since the ABC condition includes \eqref{eq:varbndrho2} (consequence of our expected smoothness~\eqref{eq:ER} used in our proofs) as a special case, this implies that the assumptions in our Theorem~\ref{theo:PLConstant} are implicitly a special case of the assumptions in Theorem 3 in~\citep{Khaled-nonconvex-2020}. As such, here we would like to draw out the similarities and differences of the results in these two theorems.

\paragraph{Our Theory is less general.}   Theorem 3 in~\citep{Khaled-nonconvex-2020} is \emph{more general} as compared to our Theorem~\ref{theo:PLConstant} since the ABC assumption~\eqref{eq:ABCpaper} includes \eqref{eq:varbndrho2} as a special case. Since our Theorem~\ref{theo:PLConstant} is less general, it is in some sense stronger, as we explain next. 

\paragraph{Anytime result.} 
Theorem 3 in~\cite{Khaled-nonconvex-2020} is not an \emph{anytime} result. One needs to fix the total number of iterations $T$ for which  SGD will run before setting the step-size. This is in contrast to our  Theorem~\ref{theo:PLConstant}, which does not depend on the final number of steps taken, and thus allows one to simply \emph{monitor} the progress of SGD and halt when a desired tolerance is reached.

\paragraph{Simple step-size.} 
This dependency on the total number of iterations $T$ appears in the proposed step-size of Theorem 3 in~\citep{Khaled-nonconvex-2020} which is also iteration dependent (our step-size is constant). As we explain below for the full batch case (deterministic gradient descent) the step-size of ~\citep{Khaled-nonconvex-2020} still depends on the PL parameter $\mu$ while our proposed step-size does not. 
 
 \paragraph{Mini-batch analysis.} 
Because of our constant step-size in Theorem~\ref{theo:PLConstant}, we are able to provide a clear mini-batch analysis and an optimal mini-batch size in  Corollary~\ref{theo:SGDforPolyakOVer}.
 
To give a simple example contrasting the two theorems, consider the full batch case ($\tau =n$) and the notation in~\cite{Khaled-nonconvex-2020}.

When $\tau =n$  the three parameters of the ABC condition~\eqref{eq:ABCpaper} in~\cite{Khaled-nonconvex-2020} are given by $A =0,$ $B=1$ and $C =0$. In this setting, the step-size  $\gamma_k$ in Theorem 3 in~\cite{Khaled-nonconvex-2020} is given by
\begin{equation*}
\gamma_k \; = \; 
\begin{cases}
\frac{1}{2L}& \quad    \mbox{if } K \leq \frac{2L}{\mu} \mbox{ or } t \leq  \left\lceil \frac{K}{2}\right\rceil \\
\frac{2}{\mu \left( 4L/\mu +k -K/2\right)}  &\quad  \mbox{if } K \geq \frac{2L}{\mu} \mbox{ and } k \geq  \left\lceil \frac{K}{2}\right\rceil \
\end{cases}
\end{equation*}
and thus depends on the iteration $k$ counter, the total number of iterations $K$, the PL parameter $\mu$ and smoothness $L$.
 Compare this to our step-size  $\gamma = 1/L$  in Theorem~\ref{theo:PLConstant} and Corollary~\ref{theo:SGDforPolyakOVer}  in the full batch case. Our step-size is thus always larger and matches the standard step-size of gradient descent.

\subsection{Proof of Theorem~\ref{theo:SGDforPolyakOVer} }
\begin{proof}
By Theorem~\ref{theo:hierarchy} we have that the \ref{eq:expresidual} condition holds. Thus Theorem~\ref{theo:PLConstant} holds. Furthermore, since $f$ is interpolated we have that $\sigma =0,$ which when combined with Theorem~\ref{theo:PLConstant} and~\eqref{eq:itercomplexPL} gives~\eqref{eq:itercomplexPLover}. 
 
 The total complexity~\eqref{eq:totalcomplexPLover} follows by using Lemma~\ref{prop:bniceconst} and the expression for $\rho$ in~\eqref{eq:bniceconst} and plugging into~\eqref{eq:itercomplexPLover}. Since~\eqref{eq:totalcomplexPLover} is a linear function in  $b$, the minimum depends on  the sign of its slope. Taking the derivative in $b$ we have the sign slope is given by $\left(1- \frac{2\kappa_{\max}}{n-1} \right)$. If the slope is negative, we want $b$ to be a large as possible, that is $b=n$. Otherwise if the slope is positive $b=1$  is optimal.
\end{proof}

\subsection{Proof of Theorem~\ref{TheoremPLDecreasing}}
\label{sec:theoplswitch}
\begin{proof}
Let  $\gamma_k \eqdef \frac{2k+1}{(k+1)^2 \mu}$ and let $k^*$ be an integer that satisfies
\begin{equation}\label{eq:asdioaisdn}
\gamma_{k^*} \leq \frac{\mu}{L (\mu +2\rho)}.
\end{equation}
Note that $\gamma_k$ is decreasing in $k$ and  consequently $\gamma_k \leq \frac{\mu}{L (\mu +2\rho)}$ for all $k \geq k^*.$ This in turn guarantees that~\eqref{eq:subPLconvexstep}  holds for all $k\geq k^*$ with $\gamma_k$ in place of $\gamma$, that is
\begin{equation}
\Exp[f(x^{k+1})-f^*]  \leq \frac{k^2}{(k+1)^2} \Exp[f(x^k)-f^*] + \frac{L\sigma^2}{\mu^2}\frac{(2k+1)^2}{(k+1)^4 }.
\end{equation}
Multiplying both sides by $(k+1)^2$ we obtain
\begin{eqnarray*}
(k+1)^2 \Exp[f(x^{k+1})-f^*] &\leq & 
k^2  \Exp[f(x^k)-f^*] + \frac{L\sigma^2}{\mu^2} \left(\frac{2k+1}{k+1}\right)^2 \\
 &\leq & k^2 \Exp[f(x^k)-f^*] + 4 \frac{L\sigma^2}{\mu^2},
\end{eqnarray*}
where the second inequality holds because  $\frac{2k+1}{k+1} <2$. Rearranging and summing from $t= k^* \ldots k$ we obtain:
\begin{equation}
\sum_{t=k^*}^{k} \left[ (t+1)^2 \Exp[f(x^{t+1})-f^*] - t^2 \Exp[f(x^t)-f^*] \right] \leq  \sum_{t=k^*}^{k} 4 \frac{L\sigma^2}{\mu^2}. 
\end{equation}
Using telescopic cancellation gives
\[
(k+1)^2 \Exp[f(x^{k+1})-f^*] \leq  (k^*)^2 \Exp[f(x^{k^*})-f^*] + 4 \frac{L\sigma^2}{\mu^2} (k-k^*)
\]
Dividing the above by $(k+1)^2$ gives
\begin{equation}
\label{eq:cndsiu48js2}
\Exp[f(x^{k+1})-f^*] \leq  \frac{(k^*)^2}{(k+1)^2 } \Exp[f(x^{k^*})-f^*]  +\frac{4 L \sigma^2 (k-k^*)}{\mu^2(k+1)^2 }. 
\end{equation}
For $k \leq k^*$ we have that~\eqref{eq:functionTheoremExpResidual} holds, which combined with~\eqref{eq:cndsiu48js2}, gives 
\begin{eqnarray}
 \Exp[f(x^{k+1})-f^*] &\leq &
  \frac{(k^*)^2}{(k+1)^2 } \left( 1 -  \frac{\mu^2}{(\mu+2\rho)L} \right)^{k^*} [f(x^0)-f^*] \nonumber \\ &   
 & \quad  +\frac{L \sigma^2 }{\mu^2 (k+1)^2}\left(4 (k-k^*) +   \frac{(k^*)^2 \mu^2}{ \mu+2\rho}\frac{1}{L} \right),  \label{eq:sdaiuna3}
 \end{eqnarray}
It now remains to choose $k^*$.
Choosing $k^*$ that minimizes the second line of~\eqref{eq:sdaiuna3}  gives $k^* = 2\frac{L}{\mu} \left(1+2\frac{\rho}{\mu}\right)$. With this choice of $k^*$ it is easy to show that~\eqref{eq:asdioaisdn} holds.
 Furthermore,  by using that $\frac{2}{k^*} = \frac{\mu^2}{\mu+2\rho} \frac{1}{L}$ in~\eqref{eq:sdaiuna3} gives
\begin{eqnarray}
 \Exp[f(x^{k+1})-f^*] &\leq &
 \frac{( k^*)^2}{(k+1)^2 } \left( 1 - \frac{2}{k^*}  \right)^{k^*} [f(x^0)-f^*] \nonumber \\ 
 &   & \quad  +\frac{L \sigma^2 }{\mu^2 (k+1)^2}\left(4 (k-k^*) +   2k^* \right) \nonumber \\
  & \leq &  \frac{( k^*)^2}{(k+1)^2 e^2}  [f(x^0)-f^*]  +\frac{2 L \sigma^2 }{\mu^2 (k+1)^2}\left(2 k-k^* \right)  \nonumber \\
  & \leq &\frac{( k^*)^2}{(k+1)^2 e^2}  [f(x^0)-f^*]  +\frac{4 L \sigma^2 }{\mu^2 (k+1)}. 
\end{eqnarray}
where in the second inequality we have used that $\left( 1 -  \frac{1}{x} \right)^{2x} \leq \frac{1}{e^2}$  for all $x \geq 1,$ and in the third inequality we used that
$\frac{2 k-k^*}{k+1} \leq \frac{2 k}{k+1}  \leq 2. $
\end{proof}

\subsection{Proofs of Section~\ref{sec:examples} }

\subsubsection{Proof of Theorem~\ref{thm:BV-WS-ES}}
We repeat the statement of this theorem detailing as well the exact constants for each assumption.

\begin{theorem} \label{thmapp:BV-WS-ES}The following hierarchy holds
\vspace{-0.15cm}\begin{center}
\begin{tikzcd}
\boxed{ \ref{eq:BV}(\sigma^2)+ \ref{eq:WS}(L)} \arrow[r] &
  \boxed{ES(L)} \arrow[r] &
  \boxed{\colorbox{blue!20}{ER}(L)}
\end{tikzcd}
\end{center}
Furthermore, there exist functions for which~\ref{eq:ER} condition holds but~\eqref{eq:BV} does not. 
\end{theorem} 
\begin{proof}
We will prove the alternative form of the ER condition given in~\eqref{eq:varbndrho2} and the ES condition given further down in~\eqref{eq:varbndcL2}. Now assuming that \eqref{eq:BV} and \eqref{eq:WS} hold we have that
\begin{align*}
\E{\norm{g(x)}^2} & = \E{\norm{g(x) - \nabla f(x) + \nabla f(x)}^2}  \\
&\leq 2 \E{\norm{g(x) - \nabla f(x) }^2} + 2\norm{\nabla f(x)}^2 \\
& \leq  4L (f(x) -f(x^*)) +2 \sigma^2.
\end{align*}
This shows that if~\eqref{eq:BV} holds with constant $\sigma^2$ and \eqref{eq:WS} with constant $L$ then the Expected Smoothness bound~\eqref{eq:varbndcL2} holds with constant $L.$ The fact that $ES(L)$ implies $ER(L)$ follows since
\[\E{\norm{g(x)}^2}  \leq 4L (f(x) -f(x^*)) +2 \sigma^2 \leq 4L (f(x) -f(x^*)) +2 \sigma^2 + \norm{\nabla f(x)}^2,\]
which shows that~\eqref{eq:varbndrho2} holds with $\rho = L.$ 

As an example for which the ER condition holds but~\ref{eq:BV} does not, take any smooth and strongly convex function such as $f(x) = \frac{1}{2n}\norm{Ax-b}^2 =\frac{1}{2n}\sum_{i=1}^n (A_{i}^\top x-b_i)^2 $ where $A\in\R^{d\times d}$ is invertible and $A_i$ is the ith row of $A$. Let $g(x) =\frac{1}{n} A_i (A_i^\top x-b)$ where $i$ is chosen uniformly at random. It is now easy to show that $\norm{g(x)}^2$ is unbounded. In contrast, we know that $g(x)$ satisfies the ER condition due to Theorem~\ref{theo:hierarchy} since strong convexity implies convexity which implies $x^*$--convexity.
 \end{proof}
\subsubsection{Separable, smooth and PL.} Let
$f(x) = \frac{1}{n} \sum_{i=1}^n f_i(x_i),$ which are smooth and interpolated. If in addition each $f_i(x_i)$ satisfies the PL condition with constant $\mu_i$ then there exists $x^* \in \cX^*$ such that $f(x)$ satisfies the PL condition with $\mu =\min_{i=1,\ldots, n} \frac{\mu_i}{n} .$ Indeed since
\begin{align*}
\norm{\nabla f(x)}^2 & = \sum_{i=1}^n\frac{1}{n^2} \norm{\nabla f_i(x_i)}^2 \\
& \geq  \sum_{i=1}^n\frac{\mu_i}{n^2} (f_i(x_i) - f_i(x^*)) \\
& \geq  \min_{i=1,\ldots, n} \frac{\mu_i}{n} (f(x) - f(x^*)). 
\end{align*}

\subsubsection{Proof of Lemma~\ref{lem:nonlsqPLexe} }
Consider the problem

\begin{equation}\label{eq:nonlinearsqr}
\min_{x \in \R^d} f(x) \eqdef \frac{1}{2n} \norm{F(x) -y}^2=\frac{1}{2n}  \sum_{i=1}^n (F_i(x) -y_i)^2 
\end{equation}
 where $y \in \R^n.$

\begin{proof}
The Jacobian of $F$ is given by
$DF(x)^\top = [\nabla F_1(x), \ldots, \nabla F_n(x)] \in \R^{d \times n}$.
Note that
\begin{align}
\nabla f(x) & = \; \frac{1}{n}DF(x)^\top (F(x) -y),\\
\nabla f_i(x) & = \; \nabla F_i(x) (F_i(x) -y_i).
\end{align}
Consequently $\nabla f(x^*) = \nabla f_i(x^*) =0.$ Finally, we  suppose that the $ F_i(x)$ functions are Lipschitz  and the Jacobian $DF(x)$ has full row rank, that is,
\begin{align}
 \norm{\nabla F_i(x) } & \leq \; u, \quad \forall i \in \{1,\ldots, n\}, \forall k. \label{eq:boundedgrads}\\
 \norm{DF(x) ^\top v} &  \geq \;  \ell \; \norm{v}, \quad \forall v . \label{eq:jacolower}
\end{align}
Under these assumptions, our objective~\eqref{eq:nonlinearsqr} satisfies the PL condition and the expected smoothness condition.
Indeed~\eqref{eq:expsmooth} holds using
\begin{align}
\EE{i}{\norm{\nabla f_i(x) -\nabla f_i(x^*) }^2} &= \EE{i}{ \norm{\nabla f_i(x) }^2} =  \frac{1}{n} \sum_{i=1}^n\norm{\nabla F_i(x)}^2  (F_i(x) -y_i)^2 \nonumber \\
& \leq \frac{u}{n} \sum_{i=1}^n  (F_i(x) -y_i)^2\;  =  \frac{u}{n} \norm{F(x) -y}^2 \nonumber \\
& = \; 2u (f(x) -f(x^*)),
\end{align}
where we used~\eqref{eq:boundedgrads} in the inequality. This shows that the expected smoothness condition hold with $u = \cL.$

By using the lower bound~\eqref{eq:jacolower}  we have that
\begin{align*}
\norm{\nabla f(x) }^2 & = \frac{1}{n^2}\norm{DF(x)^\top (F(x) -y)} \\
& \geq \frac{1}{n^2}\norm{F(x)-y}^2 \min_{v}  \frac{\norm{DF(x)^\top v}^2}{\norm{v^2}} \\
& \geq \ell f(x) = \ell( f(x) - f(x^*)), 
\end{align*}
which shows that the PL condition holds with $\mu = \ell.$
\end{proof}

 The condition~\eqref{eq:jacolower} is hard to verify, and somewhat unlikely to hold for all $x \in \R^d.$  
 Though if we had consider a closed and bounded constraint $\cX \subset \R^d$, and applied the projected SGD method, then~\eqref{eq:boundedgrads} is more likely to hold. For instance, assuming that~\eqref{eq:jacolower} holds in neighborhood of the solution is the typical assumption used to prove the assymptotic convergence of the Gauss-Newton method (see Theorem~10.1 in~\cite{wright1999numerical}).
 
 \section{Additional Convergence Analysis Results}
\label{AppendixTheoryAdditional}

 \subsection{Convergence of SGD for Quasar Strongly Convex functions}
 \label{sec:strongquasar}
In this section we develop specialized theorems for quasar strongly convex functions,
\begin{definition}[Quasar strongly convex]
Let $\zeta>0$ and $\lambda\geq0$. Let $x^* \in\cX^*$ . We that  $f$ is $(\zeta, \mu)$- quasar strongly convex with respect to $x^*$ if for all $x \in \R^n$,
\begin{equation}
\label{eq:quasar-convex-strongly}
f(x^*) \geq f(x) +\frac{1}{\zeta} \langle \nabla f(x), x^*-x\rangle + \frac{\lambda}{2}\|x^*-x\|^2.
\end{equation}
For shorthand we write $f \in QSC(\zeta,\lambda)$.  

Note that If~\eqref{eq:quasar-convex-strongly} holds with $\lambda=0$ we say that the function $f$ is $\zeta$-quasar-convex and we write $f \in QC(\zeta)$ (same to \eqref{eq:quasar-convex} from the main paper).
\end{definition}

\subsubsection{Constant Step-size}
When $\lambda =1$ we say that $f$ is  star-strong convexity, but it is also known in the literature as quasi-strong convexity \cite{Necoara-2018, gower2019sgd} or weak strong convexity \cite{karimi2016linear}. Star-strong convexity is also used in \cite{Necoara-2018} for the analysis of gradient and accelerated gradient descent and in \cite{gower2019sgd} for the analysis of SGD.  
The following theorem is a generalization of Theorem 1 in~\cite{gower2019sgd} to quasar strongly convex functions and under the assumption of expected residual.
\begin{theorem}\label{theo:strcnvlin}
Let $\zeta>0$. Assume $f$ is $(\zeta, \lambda)$-quasar-strongly convex with respect to $x^*$ and $g \in ER(\rho)$.
Choose   $\gamma^k=\gamma \in (0, \frac{\zeta}{\gamma(2\rho + L)})$ for all $k$. Then iterates of  SGD  given by (\ref{eq:sgdstep}) satisfy: 
\begin{equation}\label{eq:convsgd}
\mathbb{E} \| x^k - x^* \|^2 \leq \left( 1 - \gamma \zeta \lambda \right)^k \| x^0 - x^* \|^2 + \frac{2 \gamma \sigma^2}{\zeta \lambda}. 
\end{equation}
\end{theorem}

\begin{proof}
Let $r^k = x^k -x^*$. From (\ref{eq:sgdstep}), we have 
\begin{align*}
\label{najxs}
\| r^{k+1}  \|^2 &\overset{\eqref{eq:sgdstep}}{ =} \; \|  x^k -x^* -\gamma \nabla f_{v^k}(x^k) \|^2\notag\\
&=\; \|  r^k \|^2 - 2\gamma \langle  r^k, \nabla f_{v^k}(x^k) \rangle + \gamma^2 \|\nabla f_{v^k} (x^k)\|^2. \notag
\end{align*}
Taking expectation conditioned on $x^k$ we obtain:
\begin{align*}
\Exp_{\cD}{\|r^{k+1}\|^2} = & \; \| r^k \|^2 - 2\gamma \langle r^k, \nabla f(x^k) \rangle + \gamma^2 \Exp_{\cD}\|\nabla f_{v^k} (x^k)\|^2\\
\overset{\eqref{eq:quasar-convex-strongly}}{ \leq } & \; (1- \gamma \zeta \lambda) \| r^k \|^2 - 2\zeta\gamma [f(x^k)-f(x^*)] + \gamma^2 \Exp_{\cD}\|\nabla f_{v^k} (x^k)\|^2.
\end{align*}
Taking expectations again and using \eqref{eq:varbndrho3}
\begin{align*}
\Exp{\|r^{k+1}\|^2}
\overset{\eqref{eq:varbndrho3}}{ \leq } & \;(1- \gamma \zeta \lambda) \| r^k \|^2 - 2\zeta\gamma [f(x^k)-f(x^*)] + \gamma^2  2(2\rho + L)(f(x) - f(x^*) +2 \gamma^2 \sigma^2\\
\leq & \; (1- \gamma \zeta \lambda) \Exp \| r^k \|^2  + 2\gamma (\gamma(2\rho + L)- \zeta)  \Exp [f(x^k)-f(x^*)] + 2 \gamma^2 \sigma^2\\
\leq &\;(1- \gamma \zeta \lambda) \Exp \| r^k \|^2 + 2\gamma^2\sigma^2,
\end{align*}
where we used in the last inequality that $\gamma(2\rho + L)\leq \zeta$ since $\gamma \leq \frac{\zeta}{(2\rho + L)} .$
Recursively applying the above and summing up the resulting geometric series gives
\begin{eqnarray}
\mathbb{E} \|r^k\|^2 &\leq &  \left( 1 - \gamma \zeta \lambda \right)^k \|r^0\|^2 + 2\sum_{j=0}^{k-1} \left( 1 -\gamma \zeta \lambda \right)^j \gamma^2\sigma^2 \nonumber\\
&\leq &   \left( 1 - \gamma \zeta \lambda \right)^k \|r^0\|^2 + \frac{2\gamma \sigma^2}{ \zeta \lambda}.\label{eq:la9ja38jf}
\end{eqnarray}
\end{proof}

\subsubsection{Stochastic Polyak Step-size (SPS)}
\begin{theorem}
\label{theo:SPSquasarstronglyconvex}
Assume interpolation~\ref{ass:over} holds.  Let all $f_i$ be $L_i$-smooth and $(\zeta, \lambda)$-quasar strongly convex functions~\eqref{eq:quasar-convex} with respect to  $x^*\in\cX^*.$\footnote{This implies that function $f(x)=\sum_{i=1}^n f_i(x)$ is also $(\zeta, \lambda)$- strongly quasar-convex function with respect to  $x^*\in\cX^*$ (see \cite{hinder2019near}).}
 SGD with $\text{SPS}$ with $c =1/2\zeta$ converges as: 
\begin{eqnarray}\label{eq:SPSquasarstronglyconvex}
\Exp\|x^{k}-x^*\|^2 \leq\left(1-\frac{\lambda \zeta}{ \E{L_v}} \right)^k\|x^0-x^*\|^2 .
\end{eqnarray}
\end{theorem}
\begin{proof}
\begin{eqnarray}
\label{eq:sc-inter}
\|x^{k+1}-x^*\|^2 &=&\|x^k-\gamma_k \nabla f_v(x^k)-x^*\|^2\notag\\
&=&\|x^k-x^*\|^2-2 \gamma_k \langle x^k-x^*, \nabla f_v(x^k) \rangle + \gamma_k^2 \| \nabla f_v(x^k)\|^2\notag\\
&\overset{\eqref{eq:quasar-convex-strongly}}{\leq}&(1-\lambda \gamma_k) \|x^k-x^*\|^2-2 \zeta \gamma_k \left[f_v(x^k)-f_v^*\right] + \gamma_k^2 \| \nabla f_v(x^k)\|^2\notag\\
&\overset{\eqref{SPLRv}}{=}&(1-\lambda \gamma_k) \|x^k-x^*\|^2-2 \zeta \gamma_k \left[f_v(x^k)-f_v^*\right] +  \frac{\gamma_k }{c} \left[f_i(x^k)-f_v^*\right] \notag\\
&=&(1-\lambda \gamma_k) \|x^k-x^*\|^2+\gamma_k \left(  \frac{1 }{c}-2 \zeta  \right) \left[f_v(x^k)-f_v^*\right]  \notag\\
&\overset{c\geq1/2\zeta}{ =} &(1-\lambda \gamma_k) \|x^k-x^*\|^2.
\end{eqnarray}
Since $f_i$ is $L_i$--smooth and $v\in \R^n_+$ we have that  $f_v$ is $L_v$--smooth where $L_v \eqdef \frac{1}{n}\sum_{i=1}^n v_i L_i.$ Consequently, 
taking expectation condition on $x^k$
\begin{eqnarray}
\Exp_{\cD}\|x^{k+1}-x^*\|^2
&\leq &\left(1-\lambda \Exp_{\cD}[\gamma_k] \right)\|x^k-x^*\|^2  \notag\\
&\overset{\eqref{eq:a8ejh8s434}}{\leq}&\left(1-\frac{\lambda}{2 c\E{L_v}} \right)\|x^k-x^*\|^2.
\end{eqnarray}
Taking expectations again and using the tower property:
\begin{eqnarray}
\Exp\|x^{k+1}-x^*\|^2
&\leq&\left(1-\frac{\lambda}{2 c \E{L_v}} \right)\Exp\|x^k-x^*\|^2.
\end{eqnarray}
\end{proof}

\begin{corollary}\label{cor:SPSquasarstronglycomplexmini}
If $v$ is the $b$--minibatch sampling, we have that $\E{L_v} \; \leq  \; L\frac{n(b-1)}{(n-1)b}+L_{\max}\frac{n-b}{(n-1)b}.$
Consequently, by selecting $c= \frac{1}{2\zeta}$ and given $\epsilon >0$, if we take $$ k\geq \left( \frac{ L}{ \zeta \lambda}\frac{n(b-1)}{(n-1)b}+\frac{L_{\max}}{ \zeta\lambda}\frac{n-b}{(n-1)b}\right)\log\left(\frac{\|x^0-x^*\|^2}{\epsilon} \right),$$ 
steps of SGD with the SPS step size then $\Exp\|x^{k}-x^*\|^2 \leq \epsilon.$
\end{corollary}
\begin{proof}
Since the interpolation condition implies that $f_i$ is convex around $x^*$, we have 
by Proposition~\ref{prop:master_lemma}  that the expected smoothness condition~\ref{eq:expsmooth} holds with $\cL(b)$ given in~\eqref{eq:cLminisupp}.
Furthermore, we have that from Lemma E.1 in~\cite{gower2019sgd} that $\E{L_v}  \leq \cL(b).$ 
Finally, from~\eqref{eq:SPSquasarstronglyconvex} we have that the iteration complexity is given by
\[k  \geq \frac{2c \E{L_v}}{ \lambda }\log\left(\frac{\|x^0-x^*\|^2}{\epsilon} \right).\]
Plugging in $c= \frac{1}{2\zeta}$ and the upperbound~\eqref{eq:cLminisupp} for $\cL(b)$ gives the result.
\end{proof}

\subsection{Convergence Analysis Results under Expected Smoothness}
\label{sec:convexpsmooth}
Our initial objective was to study the convergence of structured non-convex problems that also satisfied the Expected Smoothness (ES) Assumption~\ref{eq:ES}. The resulting rates we obtained for quasar convex functions are virtually the same that we obtained under the ER assumption, as one can see in the following section. The same cannot be said under the PL assumption. 

Rather, as we show in Section~\ref{sec:PLandES} that when analysis PL functions that also satisfied the ES condition, we found that we were not able to obtain the best known rates for full batch gradient descent, unlike Theorem~\ref{theo:PLConstant} when using the ER condition.
Of course, by Theorem~\ref{theo:hierarchy} we know that ES implies ER, thus it is clearly possible to obtain better rates using the ES condition. But since the ER condition is also weaker than the ES condition, we chose to focus this paper on the ER condition.

\subsubsection{Quasar-Convex and Expected Smoothness}

For this next theorem we first need the following lemma.
\begin{lemma}
\label{lem:varbndcL}
Suppose $f$ satisfies the expected smoothness Assumption~\ref{eq:ES}. It follows that
\begin{align}
\label{eq:varbndcL2}
\EE{\cD}{\|g (x)\|^2 } & \leq  4  \cL ( f(x)-f^* )    +2 \sigma^2,
\end{align}
\end{lemma}
\begin{proof}
Using
\begin{align*}
\norm{g(x) }^2  & \leq 2\norm{g(x) -g(x^*)}^2  +2\norm{g(x^*)}^2,
\end{align*}
and taking the supremum  over $x^* \in \cX^*$  and expectation  together with~\eqref{eq:expsmooth}  gives the result.
\end{proof}
In our forthcoming proofs we only make use of~\eqref{eq:varbndcL2}, and as such, we will also refer to~\eqref{eq:varbndcL2} as the ES condition.

\begin{theorem} \label{theo:master-quasar-convex-ES}
Assume $f(x)$ is $\zeta-$quasar-convex \eqref{eq:quasar-convex} and $g \in ES(\cL)$. Let $0<\gamma_k<\frac{\zeta}{2\cL}$ for all $k \in \N$. Then, for every $x^* \in \cX^*$ 
\begin{eqnarray}\label{eq:master-quasar-convex-ES}
\min_{t=0,\dots,k-1}\E{f(x^t) - f(x^*)} \leq \frac{\norm{x^0 - x^*}^2}{2\sum_{i=0}^{k-1}\gamma_i(\zeta - 2\gamma_i\cL)}  + \sigma^2 \frac{\sum_{t=0}^{k-1}\gamma_t^2}{\sum_{i=0}^{k-1}\gamma_i(\zeta -2\gamma_i\cL)}, \label{eq:josj48jrso9j84}
\end{eqnarray}
Moreover:
	\begin{enumerate}
	\item if $\forall k \in \N, \; \gamma_k = \gamma \leq \frac{\zeta}{2\cL}$, then $\forall k \in \N$,
	\begin{equation}\label{eq:master-quasar-const-ES}
\min_{t=0,\dots,k-1}\E{f(x^t) - f(x^*)} \leq \frac{\norm{x^0 - x^*}^2}{2\gamma(\zeta-2\gamma\cL)k} + \frac{\gamma\sigma^2}{\zeta-2\gamma\cL},
	\end{equation}
	\item suppose algorithm \eqref{eq:sgdstep} is run for $T$ iterations. If $\forall k=0,\dots,T-1, \; \gamma_k = \frac{\gamma}{\sqrt{T}}$ with $\gamma \leq \frac{\zeta}{4\cL}$,
	\begin{equation}
\min_{t=0,\dots,T-1}\E{f(x^t) - f(x^*)} \leq \frac{\norm{x^0 - x^*}^2 + 2\gamma^2\sigma^2}{\gamma\sqrt{T}},
	\end{equation}
	\item $\forall k \in \N, \; \gamma_k = \frac{\gamma}{\sqrt{k+1}}$ with $\gamma \leq \frac{\zeta}{2\cL}$, then $\forall k \in \N$,
	\begin{equation}\label{eq:adn8ha8}
\min_{t=0,\dots,k-1}\E{f(x^t) - f(x^*)} \leq \frac{\norm{x^0 - x^*}^2 + 2\gamma^2\sigma^2(\log(k)+1)}{4\gamma(\zeta(\sqrt{k} - 1)-\gamma\cL(\log(k) + 1)} \sim O\left(\frac{\log(k)}{\sqrt{k}}\right).
	\end{equation}
	\end{enumerate}
\end{theorem}

\begin{proof}
We have:
\begin{align*}
\norm{x^{k+1} - x^*}^2 &= \norm{x^k - x^*}^2 - 2\gamma_k\langle g(x^k), x^k - x^*\rangle + \gamma_k^2\norm{g(x^k)}^2
\end{align*}
Hence, taking expectation conditioned on $x_k$, we have:
\begin{eqnarray*}
\EE{\cD}{\norm{x^{k+1} - x^*}^2} &=& \norm{x^k - x^*}^2 - 2\gamma_k\langle \nabla f(x^k), x^k - x^*\rangle + \gamma_k^2\EE{\cD}{\norm{\nabla f_{v_k}(x_k)}^2}\\
&\overset{\eqref{eq:quasar-convex}+\eqref{eq:varbndcL2}}{\leq}& \norm{x^k - x^*} - 2\gamma_k(\zeta - 2\gamma_k\cL)(f(x^k) - f^*)) + 2\gamma_k^2\sigma^2.
\end{eqnarray*}
Rearranging and taking expectation, we have
\begin{align*}
2\gamma_k(\zeta - 2\gamma_k\cL)\E{f(x^k) - f^*} \leq \E{\norm{x^k - x^*}^2} - \E{\norm{x^{k+1} - x^*}^2} + 2\gamma_k^2\sigma^2.
\end{align*}
Summing over $k =0,\ldots, t-1$ and using telescopic cancellation gives
\begin{align*}
2\sum_{k=0}^{t-1}\gamma_k(\zeta - 2\gamma_k\cL)\E{f(x_k) - f^*} \leq \norm{x^0 - x^*}^2 - \E{\norm{x^{k} - x^*}^2} + 2\sigma^2\sum_{k=0}^{t-1}\gamma_k^2.
\end{align*}
Since $\E{\norm{x^{k} - x^*}^2} \geq 0$, dividing both sides by $2\sum_{i=1}^{t}\gamma_i(\zeta - 2\gamma_k\cL)$ gives:
\begin{align*}
\sum_{k=0}^{t-1}\E{\frac{\gamma_k(\zeta - 2\gamma_k\cL)}{\sum_{i=0}^{t-1}\gamma_i(\zeta - 2\gamma_i\cL)}(f(x^k) - f^*)} \leq \frac{\norm{x^0 - x^*}^2}{2\sum_{i=0}^{t-1}\gamma_i(\zeta - 2\gamma_i\cL)} + \frac{\sigma^2\sum_{k=0}^{t-1}\gamma_k^2}{\sum_{i=0}^{t-1}\gamma_i(\zeta - 2\gamma_i\cL)}.
\end{align*}
Thus,
\begin{align*}
\min_{k=0,\dots,t-1}\E{f(x^k) - f(x^*)} \leq \frac{\norm{x^0 - x^*}^2}{2\sum_{i=0}^{t-1}\gamma_i(\zeta - 2\gamma_i\cL)} + \frac{\sigma^2\sum_{k=0}^{t-1}\gamma_k^2}{\sum_{i=0}^{t-1}\gamma_i(\zeta - 2\gamma_i\cL)}.
\end{align*}
For the different choices of step sizes:
\begin{enumerate}
\item If $\forall k \in \N, \; \gamma_k = \gamma \leq \frac{\zeta}{2\cL}$, then it suffices to replace $\gamma_k = \gamma$ in \eqref{eq:master-quasar-convex-ES}.

\item Suppose algorithm \eqref{eq:sgdstep} is run for $T$ iterations. Let $\forall k=0,\dots,T-1, \; \gamma_k = \frac{\gamma}{\sqrt{T}}$ with $\gamma \leq \frac{\zeta}{4\cL}$. Notice that since $\gamma \leq \frac{\zeta}{4\cL}$, we have $\zeta - 2\gamma\cL \leq \frac{1}{2}$. Then it suffices to replace $\gamma_k = \frac{\gamma}{\sqrt{T}}$ in \eqref{eq:master-quasar-convex-ES}.

\item Let $\forall k \in \N, \; \gamma_k = \frac{\gamma}{\sqrt{k+1}}$ with $\gamma \leq \frac{\zeta}{2\cL}$. Note that that since $\gamma_t = \frac{\gamma}{\sqrt{t+1}}$ and using the integral bound, we have that
\begin{equation}
\sum_{t=0}^{k-1}\gamma_t^2  = \gamma^2 \sum_{t=0}^{k-1} \frac{1}{t+1} \; \leq \; \gamma^2\left(\log(k) + 1 \right). \label{eq:ks94oo8s84}
\end{equation}
Furthermore using the integral bound again we have that
\begin{eqnarray}
\sum_{t=0}^{k-1}\gamma_t &\geq& 2\gamma\left(\sqrt{k} -1\right)
\label{eq:ia37ha37ha3}.
\end{eqnarray}
Now using~\eqref{eq:ks94oo8s84} and~\eqref{eq:ia37ha37ha3} we have that
\begin{eqnarray*}
\sum_{i=0}^{k-1}\gamma_i(\zeta-2\gamma_i\cL) &= &
\zeta\sum_{i=0}^{k-1}\gamma_i -2\cL \sum_{i=0}^{k-1}\gamma_i^2
\\
&\geq & 2\gamma\left( \zeta(\sqrt{k} -1) - \gamma \cL\left(\log(k) + 1 \right) \right).
\end{eqnarray*}
It remains to replace bound the sums in~\eqref{eq:master-quasar-convex-ES} by the values we have computed.
\end{enumerate}
\end{proof}

\paragraph{Specialized results for Interpolated Functions (with expected smoothness)}
Analogously to Corollary~\ref{cor:quasar-minib}, when the interpolated Assumption~\ref{ass:over} holds, we can  drop the expected smoothness assumption in Theorem~\ref{theo:master-quasar-convex-ES} in lieu for standard smoothness assumptions.

\begin{theorem}
\label{theo:SGDforStarOVer}
Let $f_i(x)$ be $L_i$-smooth for $i=1,\ldots, n$,   $f(x)$ be $\zeta-$quasar-convex \eqref{eq:quasar-convex} and assume that the interpolated Assumption~\ref{ass:over} holds. Consequently $g\in \text{ES}(\cL).$ If we use the step size
\begin{equation}
\label{ncaoikaosdover-smooth}
\gamma_k \equiv \gamma\leq  \frac{\zeta}{2\cL},
\end{equation} 
for all $k$, then SGD given by \eqref{eq:sgdstep} converges 
\begin{eqnarray}\label{eq:SGDforStarOVer}
\min_{i=1,\ldots, k}\E{f(x^i) - f^*}
\quad \leq \quad \frac{1}{k}\frac{\norm{x^0 - x^*}^2}{2\gamma(\zeta - 2\gamma\cL)}.
\end{eqnarray}
Hence, if $\gamma = \frac{\zeta}{4\cL}$ and given $\epsilon>0$ we have that
\begin{equation}
k\geq  \frac{4\cL}{\epsilon \zeta^2} \norm{x^0 - x^*}^2,
\end{equation}
implies $\min_{i=1,\ldots, t}\E{f(x^i) - f^*} \leq \epsilon.$ 
 \end{theorem}
\begin{proof}
By Theorem~\ref{theo:hierarchy} we have that the ES condition holds. Thus Theorem~\ref{theo:master-quasar-convex-ES} holds. Furthermore, since $f$ is interpolated we have that $\sigma =0,$ which when combined with 
\eqref{eq:master-quasar-const-ES} from Theorem~\ref{theo:master-quasar-convex-ES} gives the result.
\end{proof}
Specializing Theorem~\ref{theo:SGDforStarOVer} to the full batch setting, we have that gradient descent (GD) with step size $\gamma = \frac{\zeta}{4L} $ converges at a rate of $$f(x^t) - f(x^*) \leq \frac{4L\norm{x^0 - x^*}^2}{\zeta^2k},$$
where we have used that $\cL =L$ in the full batch setting and the smoothness of $f$ guarantees that the sequences $f(x^1),\ldots, f(x^t)$ for GD is a decreasing sequence. Similar to the result of the main paper on GD, we note that this is exactly the rate given recently for GD for quasar-convex functions in~\cite{Guminov2017}, with the exception that we have a squared dependency on $\xi$ the quasar-convex function.

\paragraph{Quasar-strongly convex functions and Expected smoothness}
Similar to Theorem~\ref{theo:strcnvlin} we present below the convergence of SGD with constant step-size for $(\zeta, \lambda)$-quasar-strongly convex functions under the expected smoothness.
\begin{theorem}\label{theo:strcnvlin2}
Let $\zeta>0$. Assume $f$ is $(\zeta, \lambda)$-quasar-strongly convex and that $(f,\cD)\sim ES(\cL)$.
Choose   $\gamma^k=\gamma \in (0,  \frac{\zeta}{2\cL}]$ for all $k$. Then iterates of  SGD  given by (\ref{eq:sgdstep}) satisfy: 
\begin{equation}\label{eq:convsgd2}
\mathbb{E} \| x^k - x^* \|^2 \leq \left( 1 - \gamma \zeta \lambda \right)^k \| x^0 - x^* \|^2 + \frac{2 \gamma \sigma^2}{\zeta \lambda}. 
\end{equation}
\end{theorem}
\begin{proof}
Similar to the proof of Theorem~\ref{theo:strcnvlin} but using ES instead of ER.
\end{proof}

\subsubsection{PL and Expected Smoothness.}
\label{sec:PLandES}
In this section we present four main theorems for the convergence of SGD with constant and decreasing step size. Through our results we highlight the limitations of using expected smoothness in the PL setting and explain why one needs to have the expected residual to prove efficient convergence. 

\begin{theorem}
\label{theo:SGDforPolyakOptimization}
Let $f(x)$ be $L$-smooth and PL function and that $g \in ES(\cL)$. Choose $\gamma^k=\gamma \leq \frac{\mu}{2 L \cL}$ for all $k$. Then SGD given by \eqref{eq:sgdstep} converges as follows: 
\begin{equation}\label{functionTheorem}
\Exp[f(x^{k})-f^*] \leq \left(1- \gamma \mu \right)^k [f(x^0)-f^*] + \frac{ L \gamma \sigma^2} {\mu}
\end{equation}
 \end{theorem}

\begin{proof}
By Proposition~\ref{prop:hierarchyapp} we have that the expected smoothness condition holds with $\rho = \cL -\mu.$ Thus by Theorem~\ref{theo:PLConstant} we have that with a step size
\[\gamma_k = \gamma\leq \frac{1}{1 +2\rho/\mu} \frac{1}{L} = \frac{1}{1 +2( \cL -\mu)/\mu} \frac{1}{L} =  \frac{\mu}{2 L \cL}\]
the iterates converge according to~\eqref{functionTheorem}.
\end{proof}

\paragraph{Limitation of Theorem~\ref{theo:SGDforPolyakOptimization}.} 
Let us consider the case where $|S|= n$ with probability one. That is, each iteration~\eqref{eq:sgdstep} uses the full batch gradient. Thus $\sigma =0$ and the expected smoothness parameter becomes $\cL =L$. Consequently, from Theorem~\ref{theo:SGDforPolyakOptimization} we obtain: 
\begin{eqnarray}
\Exp[f(x^{k})-f^*]  &\leq& \left(1- \gamma \mu \right)^k [f(x^0)-f^*].
\end{eqnarray}
For $\gamma=\frac{\mu}{2 L \cL}$ (larger possible value) the rate of the gradient descent is $\rho=1- \frac{\mu^2}{2 L^2}$.Thus
the resulting iteration complexity (number of iterations to achieve given accuracy) for gradient descent becomes $k \geq 2L^2/\mu^2$. However it is known that for minimizing PL functions, the iteration complexity of gradient descent method is  $k \geq 2L/\mu$. Thus the result of Theorem~\ref{theo:SGDforPolyakOptimization} give as a suboptimal convergence for gradient descent and the gap between the predicted behavior and the known results could potentially be very large.

\subsection{Minibatch Corollaries without Interpolation}
\label{secapp:minibatch}
In this section we state the corollaries for the main theorems when $v$ is a $b$ minibatch sampling. Differently than what we did in the main paper, we will not assume interpolation. Instead, we will use the weaker assumptions that each $f_i$ is $x^*$--convex. 
\subsubsection{Quasar Convex}
\begin{corollary} \label{cor:quasar-minib-strong} Assume $f$ is $\zeta$-quasar-strongly convex  and that each 
$f_i$ is $L_i$--smooth and $x^*$--convex. If $v$ is a $b$-minibatch sampling and $\gamma_k \equiv \frac{1}{2} \frac{\zeta}{(2\rho + L)}$ then,
		\begin{equation}\label{eq:master-quasar-const-b}
\min_{t=0,\dots,k-1}\E{f(x^t) - f(x^*)} \leq 2\norm{x^0 - x^*}^2\frac{2L_{\max} \frac{n-b}{(n-1)b} + L}{ \zeta^2 k} + \frac{\frac{1}{b} \frac{n-b}{n-1} \sigma_1^2 }{2L_{\max} \frac{n-b}{(n-1)b} + L} .
	\end{equation}
\end{corollary}
\begin{proof}
By Theorem~\ref{theo:hierarchy} we have that the \ref{eq:expresidual} condition holds. Thus, the main Theorem~\ref{theo:master-quasar-convex-res} holds. Replacing the constants $\rho$ and $\sigma$ by their corresponding minibatch constants in~\eqref{eq:bniceconst} gives the result.
\end{proof}	

\subsubsection{PL Function}
\begin{corollary} \label{cor:PLConstantmini} Let $b \in \{1,\ldots, n\}$ and let $v$ be a  $b$--minibatch sampling with replacement.  Furthermore let each $f_i$ be $L_{i}$--smooth and convex around $x^*$. If $f$ satisfies the PL condition~\eqref{eq:PL}, then by Theorem~\ref{theo:PLConstant} 
if
\begin{equation} \label{eq:gammamini}
\gamma = \frac{ \mu (n-1)b}{\displaystyle \mu (n-1)b +2L_{\max}(n-b)} \frac{1}{L},
\end{equation}
then
 \begin{align}\label{eq:suboptmini}
 \E{f(x^{k}) -f^*} &  \leq  \left( 1- \frac{ \mu^2 (n-1)b}{\displaystyle \mu (n-1)b +2L_{\max}(n-b)} \frac{1}{L}\right)^{k}(f(x^0) -f^*) \nonumber \\
 & \quad \quad + \frac{n-b}{\displaystyle \mu (n-1)b +2L_{\max}(n-b)} \sigma^2 .
 \end{align}
 \end{corollary}
\begin{proof}
The proof follows by plugging in the values of $\rho$ and $\sigma$ given in Proposition~\ref{prop:bniceconst} into~\eqref{eq:gammaSGDPLcompl} and~\eqref{eq:functionTheoremExpResidual}.
\end{proof}

\end{document}